\documentclass[12pt]{article}
\usepackage[reqno]{amsmath}
\usepackage{amsthm}
\usepackage{amssymb,amscd,latexsym}
\usepackage{paralist}
\usepackage{bm}
\usepackage[bbgreekl]{mathbbol}
\usepackage[all]{xy}
\newcommand{\A}{{\mathbb{A}}}
\newcommand{\C}{{\mathbb{C}}}
\newcommand{\F}{{\mathbb{F}}}
\newcommand{\G}{\mathbb{G}}
\newcommand{\Pa}{{\mathbb{P}}}
\newcommand{\Q}{{\mathbb{Q}}}
\newcommand{\oQ}{\overline{\Q}}

\newcommand{\Z}{{\mathbb{Z}}}
\newcommand{\hZ}{\hat{\Z}}

\newcommand{\hM}{\hat{M}}
\newcommand{\aM}{M^{\ast}}

\newcommand{\aX}{X^{\ast}}
\newcommand{\aF}{\F^{\ast}}
\newcommand{\hP}{\hat{P}}
\newcommand{\hR}{\hat{R}}

\newcommand{\Ann}{\mathrm{Ann}\,}
\newcommand{\Ass}{\mathrm{Ass}\,}
\newcommand{\cl}{\mathrm{cl}\,}

\newcommand{\ddet}{\mathrm{det}}

\newcommand{\ddiv}{\mathrm{div}\,}

\newcommand{\even}{\mathrm{even}}

\newcommand{\id}{\mathrm{id}}

\renewcommand{\mod}{\;\mathrm{mod}\;}

\newcommand{\odd}{\mathrm{odd}}

\newcommand{\per}{\mathrm{per}}
\newcommand{\rank}{\mathrm{rank}}
\newcommand{\rk}{\mathrm{rk}}

\newcommand{\spec}{\mathrm{spec}\,}
\newcommand{\supp}{\mathrm{supp}\,}
\newcommand{\Aut}{\mathrm{Aut}}
\newcommand{\coker}{\mathrm{coker}\,}

\newcommand{\End}{\mathrm{End}\,}

\newcommand{\GL}{\mathrm{GL}}
\newcommand{\Hom}{\mathrm{Hom}}

\newcommand{\tht}{\mathrm{ht}}
\newcommand{\Imm}{\mathrm{Im}\,}

\newcommand{\Ker}{\mathrm{Ker}\,}

\newcommand{\Tor}{\mathrm{Tor}}

\newcommand{\tr}{\mathrm{tr}}

\newcommand{\Mh}{\mathcal{M}}
\newcommand{\Nh}{\mathcal{N}}

\newcommand{\Wh}{\mathrm{Wh}}

\newcommand{\ea}{\mathfrak{a}}
\newcommand{\eb}{\mathfrak{b}}

\newcommand{\emm}{{\mathfrak{m}}}

\newcommand{\ep}{\mathfrak{p}}

\newcommand{\opi}{\overline{\pi}}

\newcommand{\oGamma}{\overline{\Gamma}}
\newcommand{\oR}{\overline{R}}
\newcommand{\ox}{\overline{x}}

\newcommand{\oz}{\overline{z}}
\newcommand{\hA}{\hat{A}}

\newcommand{\oa}{\overline{a}}

\newcommand{\opartial}{\overline{\partial}}
\newcommand{\hpartial}{\hat{\partial}}
\newcommand{\hvarepsilon}{\hat{\varepsilon}}

\newcommand{\silo}{\xrightarrow{\sim}}

\newcommand{\hullet}{{\scriptscriptstyle \bullet\,}}

\newtheorem{theorem}{Theorem}[section]
\newtheorem{lemma}[theorem]{Lemma}
\newtheorem{prop}[theorem]{Proposition}
\newtheorem{cor}[theorem]{Corollary}
\newtheorem{remark}[theorem]{Remark}
\newtheorem*{rem}{Remark}
\newenvironment{example}{\noindent {\bf Example.}}{}
\newenvironment{proofof}{\noindent {\it Proof of}}{\mbox{}\hfill$\Box$}

\begin{document}
\title{$p$-adic limits of renormalized logarithmic Euler characteristics}
\author{Christopher Deninger\footnote{supported by CRC 878}}
\date{\ }
\maketitle
\section{Introduction} \label{sec:0}
Given a countable residually finite group $\Gamma$, we write $\Gamma_n \to e$ if $(\Gamma_n)$ is a sequence of normal subgroups of finite index such that any infinite intersection of $\Gamma_n$'s contains only the unit element $e$ of $\Gamma$. Given a $\Gamma$-module $M$ we are interested in the multiplicative Euler characteristics
\begin{equation}
\label{eq:1a}
\chi (\Gamma_n , M) = \prod_i |H_i (\Gamma_n , M)|^{(-1)^i}
\end{equation}
and the limit in the field $\Q_p$ of $p$-adic numbers
\begin{equation}
\label{eq:1b}
h_p := \lim_{n\to\infty} (\Gamma : \Gamma_n)^{-1} \log_p \chi (\Gamma_n , M) \; .
\end{equation}
Here $\log_p : \Q^{\times}_p \to \Z_p$ is the branch of the $p$-adic logarithm with $\log_p (p) = 0$. Of course, neither \eqref{eq:1a} nor \eqref{eq:1b} will exist in general. We isolate conditions on $M$, in particular \textit{$p$-adic expansiveness} which guarantee that the Euler characteristics $\chi (\Gamma_n , M)$ are well defined. That notion is a $p$-adic analogue of expansiveness of the dynamical system given by the $\Gamma$-action on the compact Pontrjagin dual $X = \aM$ of $M$. Under further conditions on $\Gamma$ we also show that the renormalized $p$-adic limit in \eqref{eq:1b} exists and equals the \textit{$p$-adic $R$-torsion} $\tau^{\Gamma}_p (M)$ of $M$. The latter is a $p$-adic analogue of the Li--Thom $L^2$ $R$-torsion $\tau^{\Gamma}_{(2)} (M)$ of a $\Gamma$-module $M$ which they related to the entropy $h$ of the $\Gamma$-action on $X$. We view the limit $h_p$ in \eqref{eq:1b} as a version of entropy which values in the $p$-adic numbers and our formula
\[
h_p = \tau^{\Gamma}_p (M)
\]
as an analogue of the Li--Thom formula of \cite{LT},
\[
h = \tau^{\Gamma}_{(2)} (M)
\]
in the expansive case. The reason for this point of view is explained in section \ref{sec:5}.

Assuming the limit \eqref{eq:1b} exists, we get a certain amount of arithmetic information about the sequence of Euler characteristics $\chi (\Gamma_n, M)$. For simplicity, let us assume that $h_p \in p \Z_p$ and $p \neq 2$. Then $\exp_p (h_p)$ is defined, where $\exp_p$ is the $p$-adic exponential function. For a number $\chi \in \Q^{\times}_p$ let $\chi^{(1)}$ be the component in $1 + p \Z_p$ under the standard decomposition
\[
\Q^{\times}_p = \mu_{p-1} \times (1 + p \Z_p) \times p^{\Z} \; .
\]
Then the limit formula \eqref{eq:1b} is equivalent to the following assertion: There exists a sequence $(u_n) , u_n \in \Q^{\times}_p$ with $u_n \to 1$ such that for large enough $n$ we have
\[
\chi (\Gamma_n , M)^{(1)} = \exp_p (h_p)^{(\Gamma : \Gamma_n)} u^{(\Gamma : \Gamma_n)}_n \; .
\]
In particular
\[
\chi (\Gamma_n , M)^{(1)} \sim \exp_p (h_p)^{(\Gamma : \Gamma_n)} 
\]
in the sense that the quotient of both sides tends to $1$ in $\Q_p$. For the general case, use Proposition \ref{t3.6}.

The theory of algebraic dynamical systems $X$ gives some useful insights to the study of multiplicative Euler characteristics and their renormalized limits: One of the motivations for the present paper was the Li--Thom theorem $h = \tau^{\Gamma}_{(2)} (M)$ mentioned above. As another example, it is known that for $\Gamma = \Z^N$ only the principal prime ideals in a prime filtration for $M$ contribute to the entropy of $X$. The proof uses a positivity argument which does not transfer to the $p$-adic case. However, in the $p$-adic case the analogous assertion for the quantity $h_p$ defined by formula \eqref{eq:1b} is still true. This follows from a basic result on Serre's local intersection numbers.

Here is a short review of the individual sections. In \S\,\ref{sec:1}, \ref{sec:2} we prove that multiplicative Euler characteristics are well defined under suitable $p$-adic conditions. In \S\,\ref{sec:1} we do this for modules over augmented rings $R$ and in \S\,\ref{sec:2} we specialize to integral group rings $R = \Z \Gamma$. In \S\,\ref{sec:3} we review and extend the theory of $\log_p \det_{\Gamma}$ which we introduced in \cite{D}. It is a $p$-adic analogue of $\log \det_{\Nh\Gamma}$ where $\det_{\Nh\Gamma}$ is the Fuglede--Kadison determinant on the von~Neumann group algebra $\Nh\Gamma$. In \S\,\ref{sec:4} using $\log_p \det_{\Gamma}$ we introduce $p$-adic $R$-torsion and prove the limit formula $h_p = \tau^{\Gamma}_p (M)$. In \S\,\ref{sec:5} we review some aspects of the theory of algebraic dynamical systems and rephrase the preceding results in dynamical terms. In the final \S\,\ref{sec:6} we specialize to $\Gamma = \Z^N$ and relate $\chi (\Gamma_n , M)$ to Serre's intersection numbers. This allows us to calculate $h_p$ for $p$-adically expansive modules $M$ explicitely. We end with a remark how some of this connects to Arakelov theory on $\Pa^N_{\Z}$.

I am grateful to Arthur Bartels who gave me the oppertunity to speak at Wolfgang L\"uck's birthday conference in M\"unster. This invitation led to the present paper. I would also like to thank Laurent Fargues and Georg Tamme for helpful conversations.
\section{Multiplicative Euler characteristics} \label{sec:1}
Let $R$ be an (associative unital) ring. A (left) $R$-module $M$ is of type $FP$ (resp. $FL$) if there is an exact sequence
\begin{equation}
\label{eq:1}
0 \longrightarrow P_d \xrightarrow{\partial} \cdots \xrightarrow{\partial} P_0 \xrightarrow{\pi} M \longrightarrow 0
\end{equation}
of finitely generated projective (resp. free) $R$-modules $P_{\nu}$ for $0 \le \nu \le d$. We say that $R$ is augmented if we are given a ring homomorphism $\varepsilon : R \to \Z$. This makes $\Z$ a right $R$-module. The following fact follows from the definitions:

\begin{prop}
\label{t1.1}
For an augmented ring $R$ and an $R$-module $M$ of type $FP$ as in \eqref{eq:1} the groups
\begin{equation}
\label{eq:2}
\Tor^R_i (\Z , M) = H_i (\Z \otimes_{R} P_{\hullet})
\end{equation}
are finitely generated $\Z$-modules which vanish for $i > d$. 
\end{prop}

For a $\Z$-module $A$ and a prime number $p$ let $\hA = \varprojlim_n A / p^n A$ be its $p$-adic completion. Write $A_{p^n} = \Ker (p^n : A \to A)$ and let $(A_{p^i})$ be the projective system with transition maps $p : A_{p^i} \to A_{p^{i-1}}$ for $i \ge 1$.

\begin{lemma}
\label{t1.2}
Let $R$ be a ring with $R_p = 0$. In the situation above, write $P^+_{\hullet}$ for the acyclic complex \eqref{eq:1} with $M$ in degree $-1$. Then we have:
\begin{compactenum}[a)]
\item $H_{\nu} (P^+_{\hullet} / p^i P^+_{\hullet}) = 0$ for $\nu \neq 1$ and $(H_1 (P^+_{\hullet} / p^i P^+_{\hullet})) = (M_{p^i})$ as projective systems
\item $H_{\nu} (\hP^+_{\hullet}) = 0$ for $\nu \neq 0,1$\\
$H_0 (\hP^+_{\hullet}) = \varprojlim_i^1 M_{p^i}$\\
$H_1 (\hP^+_{\hullet}) = \varprojlim_i M_{p^i}$.
\end{compactenum}
\end{lemma}

\begin{proof}
For a projective $R$-module $P$ we have $P_p = 0$ since $P$ is a direct summand of a power of $R$ and we assumed that $R_p = 0$. Hence we have exact sequences of complexes for $i \ge 0$
\begin{equation}
\label{eq:3}
0 \longrightarrow M_{p^i} [1] \longrightarrow P^+_{\hullet} \xrightarrow{p^i} P^+_{\hullet} \longrightarrow P^+_{\hullet} / p^i P^+ \longrightarrow 0 \; .
\end{equation}
Here $M_{p^i} [1]$ is the complex with $M_{p^i}$ in degree $-1$ and which is zero elsewhere. Note the commutative diagram
\begin{equation}
\label{eq:4}
\vcenter{\xymatrix{
0 \ar[r] & M_{p^i} [1] \ar[r] \ar[d]^p & P^+_{\hullet} \ar[r]^{p^i} \ar[d]^p & P^+_{\hullet} \ar[r] \ar@{=}[d] & P^+_{\hullet} / p^i P^+_{\hullet} \ar[r] \ar@{->>}[d] & 0 \\
0 \ar[r] & M_{p^{i-1}} [1] \ar[r] & P^+_{\hullet} \ar[r]^-{p^{i-1}} & P^+_{\hullet} \ar[r] & P^+_{\hullet} / p^{i-1} P^+_{\hullet} \ar[r] & 0 \; .
}}
\end{equation}
From \eqref{eq:3} we get an exact sequence of complexes
\[
0 \longrightarrow P^+_{\hullet} / M_{p^i} [1] \xrightarrow{p^i} P^+_{\hullet} \longrightarrow P^+_{\hullet} / p^i P^+_{\hullet} \longrightarrow 0 
\]
and hence a long exact homology sequence
\begin{equation}
\label{eq:5}
\ldots \longrightarrow H_{-1} (P^+_{\hullet} / M_{p^i} [1]) \xrightarrow{p^i} H_{-1} (P^+_{\hullet}) \longrightarrow H_{-1} (P^+_{\hullet} / p^i P^+_{\hullet}) \longrightarrow 0 \; .
\end{equation}
Since $H_{\nu} (P^+_{\hullet}) = 0$ for all $\nu$, we get
\begin{align*}
H_{-1} (P^+_{\hullet} / p^i P^+_{\hullet}) & = 0 \\
H_0 (P^+_{\hullet} / p^i P^+_{\hullet}) & \silo H_{-1} (P^+_{\hullet} / M_{p^i} [1]) = 0 \; ,
\end{align*}
noting that the map $P_0 \to P_{-1} / M_{p^i} = M / M_{p^i}$ is surjective. Moreover we find
\begin{equation}
\label{eq:6}
H_1 (P^+_{\hullet} / p^i P^+_{\hullet}) \silo H_0 (P^+_{\hullet} / M_{p^i} [1]) \xrightarrow{\overset{\pi}{\sim}} M_{p^i} \; ,
\end{equation}
since the sequence
\[
P_1 \xrightarrow{\partial} P_0 \xrightarrow{\opi} P_{-1} / M_{p^i} = M / M_{p^i}
\]
has homology
\[
\Ker \opi / \Imm \partial = \pi^{-1} (M_{p^i}) / \Imm \partial = \pi^{-1} (M_{p^i}) / \Ker \pi \xrightarrow{\overset{\pi}{\sim}} M_{p^i} \; .
\]
Because of \eqref{eq:4}, in the isomorphism \eqref{eq:6} the projection from the $i$-th to the $(i-1)$-th term on the left corresponds to $p : M_{p^i} \to M_{p^{i-1}}$ on the right. Hence we have
\[
(H_1 (P^+_{\hullet} /p^i P^+_{\hullet})) \cong (M_{p^i}) 
\]
as projective systems. Finally, the long exact sequence \eqref{eq:5} implies that
\[
H_{\nu} (P^+_{\hullet} / p^i P^+_{\hullet}) \cong H_{\nu-1} (P^+_{\hullet}) = 0 \quad \text{for} \; \nu \ge 2 \; .
\]
Thus a) is proved. According to \cite[Theorem 3.5.8]{W}, for the projective system of chain complexes of abelian groups
\[
\ldots \longrightarrow P^+_{\hullet} / p^i P^+_{\hullet} \longrightarrow P^+_{\hullet} / p^{i-1} P^+_{\hullet} \longrightarrow \cdots \longrightarrow P^+_{\hullet} / p^0 P^+_{\hullet} = 0
\]
we have exact sequences for $\nu \in \Z$
\[
0 \longrightarrow \varprojlim_i\!^1 H_{\nu+1} (P^+_{\hullet} / p^i P^+_{\hullet}) \longrightarrow H_{\nu} (\hP^+_{\hullet}) \longrightarrow \varprojlim_i H_{\nu} (P^+_{\hullet} / p^i P^+_{\hullet}) \longrightarrow 0 \; .
\]
Using a), they imply the assertions in b).
\end{proof}

Let $R$ be a ring. For any $R$-module $A$ we have a natural map
\[
\varphi_A : \hR \otimes_R A \longrightarrow \hA \; , \; (e_i) \otimes a \longmapsto (e_i a) \; .
\]
Let $P$ be a finitely generated projective $R$-module. Then there is an $R$-module $Q$ such that $P \oplus Q = R^n$ as $R$-modules for some $n \ge 1$. We have $\varphi_P \oplus \varphi_Q = \varphi_{R^n}$ and since $\varphi_{R^n}$ is an isomorphism, $\varphi_P$ is an isomorphism as well.

\begin{cor}
\label{t1.3}
Let $R$ be a ring with $R_p = 0$ and let $M$ be an $R$-module of type $FP$ with resolution \eqref{eq:1}.\\
a) If $\varprojlim^1 M_{p^i} = 0$ the map
\[
\varphi_M : \hR \otimes_R M \silo \hM
\]
is an isomorphism.\\
b) If $\varprojlim^{\nu} M_{p^i} = 0$ for $\nu = 0,1$, the sequence
\begin{equation}
\label{eq:7}
0 \longrightarrow \hP_d \longrightarrow \ldots \longrightarrow \hP_0 \longrightarrow \hM \longrightarrow 0
\end{equation}
obtained from \eqref{eq:1} by $p$-adic completion and the isomorphic sequence
\[
0 \longrightarrow \hR \otimes_R P_d \longrightarrow \ldots \longrightarrow \hR \otimes_R P_0 \longrightarrow \hR \otimes_R M \longrightarrow 0
\]
are both exact. \\
c) If $M$ is as in b) and if in addition $R$ is augmented by $\varepsilon : R \to \Z$ so that $\Z_p$ becomes a right $\hR$ module via $\hvarepsilon : \hR \to \hZ = \Z_p$, we have
\begin{equation}
\label{eq:8}
\Tor^R_i (\Z , M) \otimes_{\Z} \Z_p = \Tor^{\hR}_i (\Z_p , \hM) = \Tor^{\hR}_i (\Z_p , \hR \otimes_R M) \; .
\end{equation}
\end{cor}

\begin{proof}
a) Consider the commutative diagram
\[
\xymatrix{
\hR \otimes_R P_1 \ar[r] \ar[d]^{\wr \, \varphi_{P_1}} & \hR \otimes_R P_0 \ar[r] \ar[d]^{\wr \, \varphi_{P_0}} & \hR \otimes_R M \ar[r] \ar[d]^{\varphi_M} & 0 \\
\hP_1 \ar[r] & \hP_0 \ar[r] & \hM \ar[r] & 0
}
\]
The maps $\varphi_{P_0}$ and $\varphi_{P_1}$ are isomorphisms by the above remark. The upper line is exact since $\hR \otimes_R \_ $ is right exact. The lower line is exact by Lemma \ref{t1.2} since $\varprojlim^1_i M_{p^i} = 0$. Hence $\varphi_M$ is an isomorphism.\\
b) The exactness of \eqref{eq:7} follows with the help of Lemma \ref{t1.2}. Since $\hR \otimes_R P_i = \hP_i$ and $\hR \otimes_R M = \hM$ by a), the second assertion follows.\\
c) The ring $\Z_p$ is flat over $\Z$. Hence we find
\begin{align*}
\Z_p \otimes_{\Z} \Tor^R_i (\Z , M) & = \Z_p \otimes_{\Z} H_i (\Z \otimes_R P_{\hullet}) \\
& = H_i (\Z_p \otimes_{\hR} (\hR \otimes_R P_{\hullet})) \\
& = \Tor^{\hR}_i (\Z_p , \hM) \; .
\end{align*}
For the last step note that according to b), $\hR \otimes_R P_{\hullet}$ is a resolution of $\hR \otimes_R M = \hM$ by projective $\hR$-modules.
\end{proof}

We can now prove the main result of this section which gives a $p$-adic criterion for a certain multiplicative Euler-characteristic to be defined. 

\begin{theorem}
\label{t1.4}
a) Let $R$ be a ring with $R_p = 0$ and $M$ an $R$-module of type $FP$ such that $p^i M = p^{n_0} M$ for some $n_0 \ge 0$ and all $i \ge n_0$. Then we have $\varprojlim^1_i M_{p^i} = 0$ and
\[
\hR \otimes_R M = \hM = M / p^{n_0} M \; .
\]
Now assume that in addition $R$ is augmented and $\varprojlim M_{p^i} = 0$. Then we have:\\
b) For each $i \ge 0$ the abelian group $\Tor^R_i (\Z , M)$ is finite and its $p$-primary part is annihilated by $p^{n_0}$. For $i > d$ with $d$ as in a resolution \eqref{eq:1} the groups $\Tor^R_i (\Z , M)$ vanish. In particular the multiplicative Euler characteristics
\[
\chi_R (M) = \prod_i |\Tor^R_i (\Z , M) |^{(-1)^i} \in \Q
\]
is well defined.\\
c) For any resolution \eqref{eq:1} of $M$ the sequence
\[
0 \longrightarrow \hR_{\Q} \otimes_R P_d \longrightarrow \ldots \longrightarrow \hR_{\Q} \otimes_R P_0 \longrightarrow 0
\]
is exact, where $\hR_{\Q} = \hR \otimes_{\Z} \Q = \hR \otimes_{\Z_p} \Q_p$. If $n_0 = 0$ above, i.e. if $p : M \to M$ is surjective, then even the sequence
\[
0 \longrightarrow \hR \otimes_R P_d \longrightarrow \ldots \longrightarrow \hR \otimes_R P_0 \longrightarrow 0
\]
is exact.
\end{theorem}

\begin{rem}
For an $R$-module $M$ of type $FP$ with $p : M \to M$ an isomorphism, all conditions on $M$ in Theorem \ref{t1.4} hold and the finite groups $\Tor^R_i (\Z , M)$ have vanishing $p$-primary part. For example, let $M = R^r /R^r \partial $ where $\partial \in M_r (R)$ defines an injective map $\partial : R^r \to R^r$ by right multiplication. If $R_p = 0$ then $p : M \to M$ is an isomorphism precisely if $\opartial \in \GL_r (\oR)$ where $\oR = R \otimes_{\Z} \F_p$ and $\opartial = \partial \otimes \id$. This follows from the snake lemma which gives an exact sequence:
\[
0 \longrightarrow M_p \longrightarrow \oR^r \xrightarrow{\opartial} \oR^r \longrightarrow M/ p \longrightarrow 0 \; .
\]
\end{rem}

\begin{proof}
a) The condition $p^i M = p^{n_0} M$ for $i \ge n_0$ implies that $\hM = M / p^{n_0} M$. The commutative diagram for $i \ge j \ge n_0$
\[
\xymatrix{
0 \ar[r] & M_{p^i} \ar[r] \ar[d]^{p^{i-j}} & M \ar[r]^{p^i} \ar[d]^{p^{i-j}} & p^i M \ar@{=}[r] & p^{n_0} M \ar[r] \ar@{=}[d] & 0 \\
0 \ar[r] & M_{p^j} \ar[r] & M \ar[r]^{p^j} & p^j M \ar@{=}[r] & p^{n_0} M \ar[r] & 0
}
\]
gives isomorphisms
\[
M_{p^j} / p^{i-j} M_{p^i} \silo M / p^{i-j} M \; .
\]
For $i \ge j + n_0 , j \ge n_0$ it follows that
\[
M_{p^j} / p^{i-j} M_{p^i} \silo M / p^{n_0} M \; .
\]
Hence for $j \ge n_0$ the image of $p^{i-j} M_{p^i}$ in $M_{p^j}$ is independent of $i$ if $i \ge j + n_0$. Hence $(M_{p^i})$ satisfies the Mittag--Leffler condition and by \cite[Prop. 3.5.7]{W} we therefore have
\[
\varprojlim_i\!^1 M_{p^i} = 0 \; .
\]
By Corollary \ref{t1.3} we therefore have
\[
\hR \otimes_R M = \hM = M / p^{n_0} M \; .
\]
b) According to a) and the assumptions on $M$ we have $\varprojlim_i^{\nu} M_{p^i} = 0$ for $\nu = 0,1$. The groups $\Tor^R_i (\Z , M)$ are finitely generated abelian groups by Proposition \ref{t1.1} which vanish for $i > d$. Part c) of Corollary \ref{t1.3} implies that
\[
\Tor^R_i (\Z , M) \otimes_{\Z} \Z_p = \Tor^{\hR}_i (\Z_p , \hM) = \Tor^{\hR}_i (\Z_p , M / p^{n_0} M) \; .
\]
Hence these finitely generated $\Z_p$-modules are annihilated by $p^{n_0}$ and therefore finite. Thus the rank of $\Tor^R_i (\Z , M)$ is zero and $\Tor^R_i (\Z , M)$ is a finite abelian group. Its $p$-primary part is $\Tor^R_i (\Z , M) \otimes_{\Z} \Z_p$ which is annihilated by $p^{n_0}$ as we just saw.\\
Part c) of the theorem follows from a) and Corollary \ref{t1.3} b). 
\end{proof}

We add some facts about the $p$-conditions on $M$ which we encounted.

\begin{prop}
\label{t1.5}
Let $A$ be an abelian group. The following conditions are equivalent:\\
1) $\varprojlim_i A_{p^i} = 0$\\
2) $\Hom ((\Q / \Z) (p), A) = 0$ where $(\Q / \Z) (p) = \Q_p / \Z_p$ is the $p$-primary part of $\Q / \Z$.\\
If $A$ has bounded $p$-torsion, i.e. if there is some $m_0 \ge 1$ such that $A_{p^i} = A_{p^{m_0}}$ for $i \ge m_0$, then 1), 2) hold.
\end{prop}

\begin{proof}
Since
\[
A_{p^i} = \Hom (\Z / p^i , A) = \Hom (p^{-i} \Z / \Z , A)
\]
we find that
\[
\varprojlim A_{p^i} = \Hom (\varinjlim_i p^{-i} \Z / \Z , A) = \Hom ((\Q / \Z) (p) , A) \; .
\]
Hence assertions 1) and 2) are equivalent. If $A_{p^i} = A_{p^{m_0}}$ for $i \ge m_0$, then for any element $(a_i) \in \varprojlim_i A_{p^i}$ we have $a_i = p^{m_0} a_{i+m_0} = 0$ which implies 1).
\end{proof}

\begin{prop}
\label{t1.6}
If $R$ is a (left) Noetherian ring and $M$ a finitely generated (left) $R$-module, then $\varprojlim_i M_{p^i} = 0$.
\end{prop}

\begin{proof}
The $R$-module $M$ is Noetherian and hence the ascending chain of $R$-submodules $M_p \subset M_{p^2} \subset M_{p^3} \subset \ldots$ is stationary. Now the claim follows from Proposition \ref{t1.5}.
\end{proof}

\begin{rem}
We will be interested in integral group rings $R = \Z \Gamma$. By a result of Hall they are left- and right-Noetherian for polycyclic-by-finite groups $\Gamma$ but not in general. 
\end{rem}

\begin{prop}
\label{t1.7}
Let $R$ be a ring with $R_p = 0$ and let $M$ be an $R$-module of type $FP$. There are equivalences:\\
a) $p^i M = p^{n_0} M$ for all $i \ge n_0$ if and only if $\hR \otimes_R M$ is annihilated by $p^{n_0}$. \\
b) $(p^i M)$ is stationary for large enough $i$ if and only if $\hR_{\Q} \otimes_R M = 0$.
\end{prop}

\begin{rem}
For $R = \Z [t^{\pm 1}_1 , \ldots , t^{\pm 1}_N]$ assertion b) was first observed by Br\"auer \cite{B}.
\end{rem}

\begin{proof}
a) The implication ``$\Rightarrow$'' follows from part a) of Theorem \ref{t1.4}. The natural map $\hR \to R / p^i R$ is surjective. The induced surjection $\hR / p^i \hR \to R / p^i R$ is also injective because $p$-multiplication is injective on $R$: If $(\ox_{\nu}) \in \hR$ satisfies $\ox_i = 0$ then $x_{\nu} = p^i y_{\nu}$ for all $\nu > i$ with unique $y_{\nu}$'s in $R$. Moreover $x_{\nu} \equiv x_{\nu-1} \mod p^{\nu-1} R$ implies that $y_{\nu} \equiv y_{\nu-1} \mod p^{\nu-i-1} R$ for $\nu > i$. Setting $z_{\nu} = y_{\nu+i}$ we obtain an element $(\oz_{\nu}) \in \hR$ with $p^i (\oz_{\nu}) = (\ox_{\nu})$. Using the isomorphism $\hR / p^i \hR = R / p^i R$ we can now prove the converse implication in a). We have:
\begin{align*}
M / p^i M & = (R / p^i R) \otimes_R M = (\hR / p^i \hR) \otimes_R M \\
& = (\hR \otimes_R M) / p^i (\hR \otimes_R M) \; .
\end{align*}
If $p^{n_0}$ annihilates $\hR \otimes_R M$ it therefore follows that
\[
M / p^i M = \hR \otimes_R M \quad \text{for} \; i \ge n_0 \; .
\]
Hence the natural map
\[
M / p^i M \longrightarrow M / p^{n_0} M
\]
is an isomorphism for $i \ge n_0$ and therefore $p^i M = p^{n_0} M$.\\
b) follows from a) since modules of type $FP$ are finitely generated.
\end{proof}

\begin{prop}
\label{t1.8}
a) Let $R$ be a ring and $0 \to M' \to M \to M'' \to 0$ a short exact sequence. If $M'$ and $M''$ are of type $FP$ (resp. $FL$) with resolutions $F'_{\hullet} \to M'$ and $F''_{\hullet} \to M''$, then there is a commutative diagram with exact lines and columns where $F_i = F'_i \oplus F''_i$
\[
\xymatrix{
0 \ar[r] & F'_{\hullet} \ar[r]^i \ar[d] & F_{\hullet} \ar[r]^{\pi} \ar[d] & F''_{\hullet} \ar[r] \ar[d] & 0 \\
0 \ar[r] & M' \ar[r] \ar[d] & M \ar[r] \ar[d] & M'' \ar[r] \ar[d] & 0 \\
 & 0 & 0 & 0 & 
}
\]
Here $i$ and $\pi$ are the natural inclusion and projection. In particular $M$ is of type $FP$ (resp. $FL$) as well.\\
b) Now assume that in addition $R_p = 0$ and that $p^n M' = p^{n'_0} M'$ for $n \ge n'_0$ and $p^n M'' = p^{n''_0} M''$ for $n \ge n''_0$. Then we have that $p^n M = p^{n'_0 + n''_0} M$ for $n \ge n'_0 + n''_0$.\\
c) For any exact sequence $0 \to M' \to M \to M'' \to 0$ of $R$-modules the conditions $\varprojlim_i M'_{p^i} = 0$ and $\varprojlim_i M''_{p^i} = 0$ imply that $\varprojlim_i M_{p^i} = 0$.
\end{prop}

\begin{proof}
a) is the Horseshoe Lemma 2.2.8 of \cite{W}.\\
b) Note first that $M$ is of type $FP$ by a). The remaining assertion follows easily using Proposition \ref{t1.7} a) or directly as follows: For $n \ge n_0 = n'_0 + n''_0$ consider $x \in p^{n_0} M , x = p^{n_0} x_1$ with $x_1 \in M$. The image of $p^{n''_0} x_1$ in $M''$ is of the form $p^{n-n'_0} x''$ for some $x'' \in M$ since $n - n'_0 \ge n''_0$. Let $y$ be a preimage of $x''$ in $M$. Then $p^{n''_0} x_1 - p^{n-n'_0} y$ lies in $M'$ and hence $x - p^n y \in p^{n'_0} M' = p^n M'$. 
Thus $x \in p^n M$ as claimed. \\
Assertion c) follows since both $\varprojlim_i$ and $M \mapsto M_{p^i}$ are left exact functors.
\end{proof}
\section{Homology of discrete groups for $p$-adically expansive modules} \label{sec:2}
In this section we consider the case where $R = \Z \Gamma$ is the integral group ring of a discrete group $\Gamma$. We are particularly interested in the case where $\Gamma$ is residually finite. The group ring $\Z \Gamma$ is augmented by $\varepsilon : \Z \Gamma \to \Z$ defined by $\varepsilon (\gamma) = 1$ for all $\gamma \in \Gamma$. By definition
\[
H_i (\Gamma , M) = \Tor^{\Z \Gamma}_i (\Z , M) \; ,
\]
for a $\Gamma$- or equivalently $\Z \Gamma$-module $M$. There is a natural isomorphism:
\begin{equation}
\label{eq:9}
\widehat{\Z\Gamma} = \Big\{ \sum_{\gamma \in \Gamma} x_{\gamma} \gamma \mid x_{\gamma} \in \Z_p \; \text{with} \; |x_{\gamma}| \to 0 \; \text{for} \; \gamma \to \infty \Big\} \; .
\end{equation}
Here $|x_{\gamma}|$ is the $p$-adic absolute value and $\gamma \to \infty$ means convergence in the cofinite topology of $\Gamma$: For all $\varepsilon > 0$ there is a finite subset $S \subset \Gamma$ such that $|x_{\gamma}| < \varepsilon$ for all $\gamma \in \Gamma \setminus S$. The augmentation map $\hvarepsilon : \widehat{\Z\Gamma} \to \Z_p$ induced by $\varepsilon : \Z \Gamma \to \Z$ sends $\sum_{\gamma \in \Gamma} x_{\gamma} \gamma$ to the convergent series $\sum_{\gamma \in \Gamma} x_{\gamma}$. We set 
\[
c_0 (\Gamma) := \widehat{\Z\Gamma} \otimes_{\Z} \Q = \Big\{ \sum_{\gamma \in \Gamma} x_{\gamma} \gamma \mid x_{\gamma} \in \Q_p  \; \text{with} \; |x_{\gamma}| \to 0 \; \text{for} \; \gamma \to \infty \Big\} \; . 
\]
We call a $\Z \Gamma$-module $M$ of type $FL$ \textit{$p$-adically expansive} if $c_0 (\Gamma) \otimes_{\Z\Gamma} M = 0$ or equivalently if the flag $(p^i M)$ is stationary, see Proposition \ref{t1.7} b). More precisely, we call such a module $M$ \textit{$p$-adically expansive of exponent $n_0$} if $p^i M = p^{n_0} M$ holds for all $i \ge n_0$. By Proposition \ref{t1.8}, given an exact sequence of $\Z\Gamma$-modules
\[
0 \longrightarrow M' \longrightarrow M \longrightarrow M'' \longrightarrow 0
\]
with $M' , M''$ $p$-adically expansive of exponents $n'_0$ and $n''_0$, the module $M$ is $p$-adically expansive of exponent $n'_0 + n''_0$. The module $M$ is $p$-adically expansive of exponent zero i.e. $pM = M$ if and only if $\widehat{\Z\Gamma} \otimes_{\Z\Gamma} M = 0$, cf. Proposition \ref{t1.7} a). The reason for the name ``$p$-adically expansive'' will be explained in \S\,\ref{sec:5}. In the present context, Theorem \ref{t1.4} gives:

\begin{theorem}
\label{t2.1}
Let $\Gamma$ be a discrete group and let $\Delta$ be a normal subgroup of finite index. Let $M$ be a $p$-adically expansive $\Z\Gamma$-module of exponent $n_0$ and $FP$-resolution \eqref{eq:1} of length $d$. Assume also that $\varprojlim_i M_{p^i} = 0$. Then the homology groups $H_i (\Delta , M)$ are finite abelian groups whose $p$-primary part is annihilated by $p^{n_0}$ and which vanish for $i > d$. Moreover the sequence
\[
0 \longrightarrow c_0 (\Delta) \otimes_{\Z\Delta} P_d \longrightarrow \ldots \longrightarrow c_0 (\Delta) \otimes_{\Z\Delta} P_0 \longrightarrow 0
\]
is exact. If $n_0 = 0$ i.e. if $p : M \to M$ is surjective, then even the squence
\[
0 \longrightarrow \widehat{\Z\Delta} \otimes_{\Z\Delta} P_d \longrightarrow \ldots \longrightarrow \widehat{\Z\Delta} \otimes_{\Z\Delta} P_0 \longrightarrow 0
\]
is exact.
\end{theorem}

\begin{proof}
For any finitely generated $\Z\Gamma$-module $P$ there is an exact sequence of $\Z\Gamma$-modules for some $n$
\[
0 \longrightarrow Q \longrightarrow (\Z\Gamma)^n \longrightarrow P \longrightarrow 0 \; .
\]
If $P$ is projective, we have $P \oplus Q \cong (\Z\Gamma)^n$ as $\Z\Gamma$-modules. The $\Z \Gamma$-algebra $\Z \Gamma$ is a free $\Z\Delta$-module with basis $\gamma_1 , \ldots , \gamma_r$ a system of representatives of $\Delta \setminus \Gamma , r = (\Gamma : \Delta)$. Hence $P \oplus Q$ is a free $\Z\Delta$-module of rank $nr$ and in particular $P$ is finitely generated and projective as a $\Z\Delta$-module. Any $\Z\Gamma$-module $M$ of type $FP$ (or $FL$) i.e. with a resolution \eqref{eq:1} for $R = \Z \Gamma$ is therefore also of type $FP$ (resp. $FL$) as a $\Z\Delta$-module with the same resolution \eqref{eq:1} but now viewed as $\Z\Delta$-modules. It follows that a $p$-adically expansive $\Gamma$-module $M$ (of exponent $n_0$) is also $p$-adically expansive (of exponent $n_0$) as a $\Delta$-module. The result now follows from Theorem \ref{t1.4} applied to $M$ as an $R = \Z \Delta$-module.
\end{proof}

For a countable residually finite group $\Gamma$, sequences $\Gamma_n \to e$ as in the introduction always exist. If $M$ is a $p$-adically expansive $\Gamma$-module with $\varprojlim_i M_{p^i} = 0$ then by Theorem \ref{t2.1} the Euler characteristics
\begin{equation}
\label{eq:9a}
\chi (\Gamma_n , M) := \prod_i |H_i (\Gamma_n , M)|^{(-1)^i} \in \Q^{\times}
\end{equation}
exist for all $n$. We will later be concerned with their renormalized $p$-adic logarithmic limit as $\Gamma_n \to e$. 

We now discuss the assumptions on $M$ in Theorem \ref{t2.1} in the case where the $\Z\Gamma$-module $M$ is of type $FL$ with $d = 1$. There is thus a resolution
\begin{equation}
\label{eq:10}
0 \longrightarrow (\Z\Gamma)^s \xrightarrow{\partial} (\Z \Gamma)^r \longrightarrow M \longrightarrow 0
\end{equation}
and $M$ is $p$-adically expansive i.e. $c_0 (\Gamma) \otimes_{\Z\Gamma} M = 0$ if and only if the induced map
\[
c_0 (\Gamma)^s \xrightarrow{\partial} c_0 (\Gamma)^r
\]
is surjective. In this case we have $s \ge r$ as one sees by tensoring with $\Q_p$, viewed as a  $c_0 (\Gamma)$-module via the augmentation map $\hvarepsilon : c_0 (\Gamma) \to \Q_p$. The snake lemma applied to \eqref{eq:10} leads to an exact sequence
\[
0 \longrightarrow \varprojlim_i M_{p^i} \longrightarrow (\widehat{\Z\Gamma})^s \xrightarrow{\partial} (\widehat{\Z\Gamma})^r
\]
which implies that $\varprojlim M_{p^i} = 0$ if and only if $\partial : (\widehat{\Z\Gamma})^s \to (\widehat{\Z\Gamma})^r$ or equivalently $\partial : c_0 (\Gamma)^s \to c_0 (\Gamma)^r$ is injective. Replacing $c_0 (\Gamma)$ by $\widehat{\Z\Gamma}$ we get corresponding statements for $p$-adically expansive modules of degree zero. This proves the following proposition noting that $\Z\Gamma$ is a subring of $c_0 (\Gamma)$. 

\begin{prop} \label{t2.2}
In the situation \eqref{eq:10} the $\Z\Gamma$-module $M$ is $p$-adically expansive (of degree zero) and satisfies $\varprojlim_i M_{p^i} = 0$ if and only if $r = s$ and $\partial \in M_r (\Z \Gamma) \cap \GL_r (c_0 (\Gamma))$ (resp. $\partial \in M_r (\Z \Gamma) \cap \GL_r (\widehat{\Z\Gamma})$). 
\end{prop}

A ring $R$ is called directly finite if $ab = 1$ in $R$ implies $ba = 1$. If the ring $M_r (c_0 (\Gamma))$ is directly finite, any surjective $c_0 (\Gamma)$-linear map $c_0 (\Gamma)^r \to c_o (\Gamma)^r$ is also injective. In an exact sequence
\[
0 \longrightarrow (\Z\Gamma)^r \xrightarrow{\partial} (\Z\Gamma)^r \longrightarrow M \longrightarrow 0
\]
with $p$-adically expansive $M$ the condition $\varprojlim_i M_{p^i} = 0$ is therefore automatic if $M_r (c_0 (\Gamma))$ is directly finite. In the somewhat analogous case of the real or complex $L^1$-group algebra of $\Gamma$ it is known that $M_r (L^1 (\Gamma))$ is directly finite. This follows from a result of Kaplansky \cite{K}, p. 122. For residually finite groups this is easy to see and the same proof works in the $p$-adic case, see \cite[section 7.4]{B}:

\begin{theorem}[Br\"auer] \label{t2.3}
For a residually finite group $\Gamma$ the ring $M_r (c_0 (\Gamma))$ is directly finite for each $r \ge 1$.
\end{theorem}

\begin{proof}
For a cofinite normal subgroup $\Delta$ of $\Gamma$ set $\oGamma = \Gamma / \Delta$ and consider the natural ring homomorphism $c_0 (\Gamma) \to c_0 (\oGamma) = \Q_p [\oGamma]$ sending $x = \sum_{\gamma} x_{\gamma} \gamma$ to $\ox = \sum_{\gamma} x_{\gamma} \gamma \Delta = \sum_{\delta \in \oGamma} y_{\delta} \delta$ where $y_{\delta} = \sum_{\gamma \in \delta} x_{\gamma}$. We claim that the ring homomorphism:
\[
\alpha : c_0 (\Gamma) \longrightarrow \prod_{\Delta} c_0 (\Gamma / \Delta)
\]
is injective, where the product runs over all $\Delta$ as above. Assume $\ker \alpha \neq 0$. Since $\ker \alpha$ is an ideal there is then an element $x \in \ker \alpha$ with $x_e \neq 0$. By definition of $c_0 (\Gamma)$ the set $S \subset \Gamma$ of $e \neq \gamma \in \Gamma$ with $|x_{\gamma}| \ge |x_e| > 0$ is finite. Since $\Gamma$ is residually finite there is a cofinite normal subgroup $\Delta$ of $\Gamma$ with $\Delta \cap S = \emptyset$. Hence we have $|x_{\gamma}| < |x_e|$ for all $e \neq \gamma \in \Delta$. By the ultrametric inequality it follows that
\[
\Big| \sum_{\gamma \in \Delta} x_{\gamma} \Big| = |x_e| \neq 0 \; .
\]
On the other hand $\sum_{\gamma \in \Delta} x_{\gamma} = 0$ since $x$ is mapped to zero in $c_0 (\Gamma / \Delta)$. This is a contradiction and hence $\alpha$ is injective. It follows that the natural map
\[
M_r (c_0 (\Gamma)) \longrightarrow \prod_{\Delta} M_r (c_0 (\Gamma / \Delta))
\]
is injective as well, \cite[Corollary 7.15]{B}. Each of the rings $M_r (c_0 (\Gamma / \Delta))$ is directly finite since $c_0 (\Gamma / \Delta)$ is finite dimensional as a $\Q_p$-vector space and $M_r (c_0 (\Gamma / \Delta))$ is naturally a subring of $\End_{\Q_p} (c_0 (\Gamma / \Delta))$. Hence $M_r (c_0 (\Gamma))$ is directly finite as well. 
\end{proof}
\section{Review of a $p$-adic determinant} \label{sec:3}
In this section which is based on \cite{D} we review the definition of a $p$-adic analogue of the Fuglede--Kadison determinant and its properties. 

Let $B$ be a $\Q_p$-Banach algebra whose norm takes values in $p^{\Z} \cup \{ 0 \}$. Let $\tr_B : B \to \Q_p$ be a trace functional i.e. a continuous linear map with $\tr_B (ab) = \tr_B (ba)$ for all $a,b \in B$. Consider the $\Z_p$-Banach algebra $A = B^0 := \{ b \in B \mid \| b\| \le 1 \}$ and let
\[
U^1 = 1 + pA = \{ b \in B \mid \| 1-b\| < 1 \}
\]
be the normal subgroup of $1$-units in $A^{\times}$. Note here that for $a \in A$ the convergent series
\[
(1 + pa)^{-1} := \sum^{\infty}_{\nu =0} (-pa)^{\nu}
\]
gives an inverse to $1 + pa \in U^1$. There is an exact sequence of groups
\begin{equation}
\label{eq:10d}
1 \longrightarrow U^1 \longrightarrow A^{\times} \xrightarrow{\pi} (A / pA)^{\times} \longrightarrow 1 \; .
\end{equation}
The projection $\pi$ is surjective since for $\oa = a + pA$ in $(A / pA)^{\times}$ there is some $b \in A$ with $ab = 1 + pc$ and $ba = 1 + bd$ with $c,d \in A$. Since $1 + pc$ and $1 + pd$ are units it follows that $a$ is a unit as well.

The logarithmic series
\[
\log : U^1 \longrightarrow A \; , \; \log u = - \sum^{\infty}_{\nu=1} \frac{(1-u)^{\nu}}{\nu}
\]
converges and defines a continuous map. In \cite[Theorem 4.1]{D} it is shown that the map
\begin{equation}
\label{eq:11}
\tr_B \log : U^1 \longrightarrow \Q_p
\end{equation}
is a homomorphism. This is a consequence of the $p$-adic Campbell--Baker--Hausdorff formula. If $\tr_B (A) \subset \Z_p$, then $\tr_B \log$ takes values in $\Z_p$.

Let $\Gamma$ be a discrete group and recall the algebra
\[
c_0 (\Gamma) = \widehat{\Z\Gamma} \otimes_{\Z_p} \Q_p = \Big\{ \sum_{\gamma} x_{\gamma} \gamma \mid x_{\gamma} \in \Q_p \; \text{with} \; |x_{\gamma}| \to 0 \; \text{for} \; \gamma \to \infty \Big\} \; .
\]
Equipped with the norm
\[
\Big\| \sum_{\gamma} x_{\gamma} \gamma \Big\| = \max_{\gamma} |x_{\gamma}|
\]
it becomes a $\Q_p$-Banach algebra and the norm topology induced on $\widehat{\Z\Gamma} = c_0 (\Gamma)^0$ is the $p$-adic topology on $\widehat{\Z\Gamma}$. A short calculation shows that the map
\[
\tr_{\Gamma} : c_0 (\Gamma) \longrightarrow \Q_p \; , \; \sum_{\gamma} x_{\gamma} \gamma \longmapsto x_e
\]
defines a trace functional. More generally, $B = M_r (c_o (\Gamma))$ with norm $\| (a_{ij})\| = \max_{i,j} \| a_{ij}\|$ is a $\Q_p$-Banach algebra with trace functional
\[
\tr_{\Gamma} : M_r (c_0 (\Gamma)) \xrightarrow{\tr} c_0 (\Gamma) \xrightarrow{\tr_{\Gamma}} \Q_p \; .
\]
The algebra $A = B^0$ is given by $M_r (\widehat{\Z\Gamma})$ and we have $U^1 = 1 + p M_r (\widehat{\Z\Gamma})$. The exact sequence \eqref{eq:10d} reads as follows here:
\begin{equation}
\label{eq:12}
1 \longrightarrow 1+p M_r (\widehat{\Z\Gamma}) \longrightarrow \GL_r (\widehat{\Z\Gamma}) \longrightarrow \GL_r (\F_p \Gamma) \longrightarrow 1 \; .
\end{equation}
The map \eqref{eq:11} is denoted by
\begin{equation}
\label{eq:13}
\log_p \ddet_{\Gamma} := \tr_{\Gamma} \log : 1 + p M_r (\widehat{\Z\Gamma}) \longrightarrow \Z_p \; .
\end{equation}
It is a continuous homomorphism of groups. We wish to extend $\log_p \det_{\Gamma}$ to $\GL_r (\widehat{\Z\Gamma})$. For this it is helpful to pass to $\GL_{\infty}$ and then to $K_1$. For each of the algebras $R = \Z\Gamma , \widehat{\Z\Gamma}, c_0 (\Gamma)$ and $\F_p \Gamma$ we set $K_T (R) = K_1 (R) / \langle \pm \Gamma \rangle$ where $\langle \rangle : R^{\times} \to K_1 (R)$ is the natural map. Note that $K_T (\Z \Gamma)$ is the Whitehead group $\Wh (\Gamma)$ of $\Gamma$ and that $K_T (\F_p \Gamma)$ is torsion if and only if $\Wh^{\F_p} (\Gamma) := K_1 (\F_p \Gamma) / \langle \F^{\times}_p \cdot \Gamma \rangle$ is torsion. The kernel $\Nh$ of the exact sequence induced by \eqref{eq:12},
\[
1 \longrightarrow \Nh \longrightarrow K_T (\widehat{\Z\Gamma}) \longrightarrow K_T (\F_p \Gamma) \longrightarrow 1
\]
is a quotient of $1 + p M_{\infty} (\widehat{\Z\Gamma})$ where $M_{\infty} (A)$ is the non-unital algebra of infinite matrices $(a_{ij})_{i,j \ge 1}$ with only finitely many non-zero entries. For residually finite groups $\Gamma$ we showed in the proof of \cite[Theorem 5.1]{D} that the homomorphism $\log_p \det_{\Gamma}$ of \eqref{eq:13} factors over $\Nh$ and conjectured that this should be true for all groups $\Gamma$. We get the following result, cf. \cite[Theorem 5.1]{D}.

\begin{theorem}
\label{t3.1}
Let $\Gamma$ be a countable residually finite discrete group such that $\Wh^{\F_p} (\Gamma)$ is torsion. Then there is a unique homomorphism
\[
\log_p \ddet_{\Gamma} : K_T (\widehat{\Z\Gamma}) \longrightarrow \Q_p
\]
with the following property: For every $r \ge 1$ the composition
\[
1 + p M_r (\widehat{\Z\Gamma}) \hookrightarrow \GL_r (\widehat{\Z\Gamma}) \longrightarrow K_1 (\widehat{\Z\Gamma}) \longrightarrow K_T (\widehat{\Z\Gamma}) \xrightarrow{\log_p \det_{\Gamma}} \Q_p
\]
coincides with the map \eqref{eq:13}.
\end{theorem}

\begin{rem}
$\Wh^{\F_p} (\Gamma)$ is torsion for torsion free elementary amenable groups by \cite{FL} and for a class of groups that comprises all word hyperbolic groups by \cite{BLR}. For $\Gamma$ as in the theorem we define the homomorphism $\log_p \det_{\Gamma}$ on $\GL_r (\widehat{\Z\Gamma})$ to be the composition
\begin{equation}
\label{eq:14}
\log_p \ddet_{\Gamma} : \GL_r (\widehat{\Z\Gamma}) \longrightarrow K_T (\widehat{\Z\Gamma}) \xrightarrow{\log_p \det_{\Gamma}} \Q_p \; .
\end{equation}
Explicitely it is given as follows: For a matrix $f \in \GL_r (\widehat{\Z\Gamma})$ there are integers $N \ge 1$ and $s \ge r$ such that in $\GL_s (\widehat{\Z\Gamma})$ we have
\[
f^N = i (\pm \gamma) \varepsilon g \; .
\]
Here $\varepsilon$ is a product of $s \times s$-elementary matrices (identity matrices with one off-diagonal entry), $g$ is in $1 + p M_s (\widehat{\Z\Gamma})$ and $i (\pm \gamma)$ for some $\gamma \in \Gamma$ is the $s \times s$-matrix $\left( \begin{smallmatrix} \pm \gamma & 0 \\ 0 & 1_{s-1} \end{smallmatrix} \right)$. Then we have
\[
\log_p \ddet_{\Gamma} f = \frac{1}{N} \log_p \ddet_{\Gamma} g = \frac{1}{N} \tr_{\Gamma} \log g \; .
\]
\end{rem}

For a countable residually finite group $\Gamma$ let $(\Gamma_n)$ be a sequence of cofinite normal subgroups with $\Gamma_n \to e$. The quotient $\Gamma^{(n)} = \Gamma / \Gamma_n$ is a finite group. For $f \in \GL_r (c_0 (\Gamma))$ let $f^{(n)}$ be its image in 
\[
\GL_r (c_0 (\Gamma^{(n)})) = \GL_r (\Q_p [\Gamma^{(n)}]) = \Aut_{\Q_p [\Gamma^{(n)}]} (\Q_p [\Gamma^{(n)}]^r)
\]
under the canonical homomorphism. We write $\det_{\Q_p} (f^{(n)})$ for the determinant of $f^{(n)}$ viewed as a $\Q_p$-linear automorphism of the $r (\Gamma : \Gamma_n)$-dimensional $\Q_p$-vector space $\Q_p [\Gamma^{(n)}]^r$. Let $\log_p : \Q^{\times}_p \to \Z_p$ be the $p$-adic logarithm normalized by the condition $\log_p (p) = 0$. 

The logarithmic determinant $\log_p \det_{\Gamma}$ which is defined for operators on generally infinite dimensional $p$-adic Banach spaces can be approximated by renormalized logarithmic determinants of operators on finite dimensional vector spaces if $\Gamma$ is residually finite. In \cite[Proposition 5.5]{D} we proved

\begin{theorem}
\label{t3.2}
For a countable residually finite discrete group $\Gamma$ and a sequence $\Gamma_n \to e$ the following formula holds if $f \in 1 + p M_r (\widehat{\Z\Gamma})$ or if $f \in \GL_r (\widehat{\Z\Gamma})$ and $\Wh^{\F_p} (\Gamma)$ is torsion:
\[
\log_p \ddet_{\Gamma} f = \lim_{n\to \infty} (\Gamma : \Gamma_n)^{-1} \log_p (\ddet_{\Q_p} (f^{(n)})) \quad \text{in} \; \Q_p \; .
\]
\end{theorem}

\begin{rem}
Note that the nature of both sides in this formula is very different. For example, for $f \in 1 + p M_r (\widehat{\Z\Gamma}) \subset \GL_r (\widehat{\Z\Gamma})$ the left hand side is defined for all $\Gamma$ by \eqref{eq:13} whereas the right hand side requires $\Gamma$ to be residually finite. We would also like to define the left hand side for all $f \in \GL_r (c_0 (\Gamma))$, so that the formula holds in this generality but we do not know how to do this in general. 
\end{rem}

\begin{cor}
\label{t3.3}
Let $\Gamma$ be a countable residually finite group. We have $\log \det_{\Gamma} f = 0$ if $f \in 1 + p M_r (\Z\Gamma)$ or if $f \in \GL_r (\Z\Gamma)$ and $\Wh^{\F_p} (\Gamma)$ is torsion.
\end{cor}

\begin{proof}
This follows from Theorem \ref{t3.2} because $f^{(n)}$ respects the $\Z$-lattice $\Z [\Gamma^{(n)}]^r$ in $\Q_p [\Gamma^{(n)}]^r$ and hence $\det_{\Q_p} (f^{(n)}) = \pm 1$.
\end{proof}

For finitely generated abelian groups $\Gamma$, Theorem \ref{t3.2} can be generalized to all $f \in \GL_r (c_o (\Gamma))$ because in this case there is the usual determinant
\[
\ddet_{c_0 (\Gamma)} : \GL_r (c_0 (\Gamma)) \longrightarrow c_0 (\Gamma)^{\times} \; .
\]
We explain this in the case where $\Gamma = \Z^N$, see also \cite[\S\,4.3]{B} for a related approach. By a classical result on Tate algebras, see e.g. \cite[Example 2.3]{D}, we have a direct product decomposition:
\begin{equation}
\label{eq:15}
c_0 (\Gamma)^{\times} = p^{\Z} \mu_{p-1} \Gamma U^1_{\Gamma} \quad \text{for} \; \Gamma = \Z^N \; .
\end{equation}
Here $\mu_{p-1} \subset \Z^{\times}_p$ is the group of $p-1$th roots of unity in $\Q^{\times}_p$ and $U^1_{\Gamma} = 1 + p \widehat{\Z\Gamma}$ is the subgroup of $1$-units in $c_0 (\Gamma)$. For $f \in \GL_r (c_0 (\Gamma))$, according to \eqref{eq:15} we may write
\begin{equation}
 \label{eq:16}
 \ddet_{c_0 (\Gamma)} (f) = p^{\nu} \zeta \gamma g \quad \text{in} \; c_0 (\Gamma)^{\times} \; \text{with} \; g \in U^1_{\Gamma} \; \text{etc.}
\end{equation}
We define
\begin{equation}
\label{eq:17}
\log_p \ddet_{\Gamma}f := \log_p \ddet_{\Gamma} g \quad \text{as in \eqref{eq:13}.}
\end{equation}
In this way we obtain a homomorphism
\[
\log_p \ddet_{\Gamma} : K_T (c_0 (\Gamma)) \longrightarrow \Q_p \; ,
\]
which extends the map in Theorem \ref{t3.1}.
This follows from the uniqueness assertion in Theorem \ref{t3.1} using the formula
\begin{equation}
\label{eq:18}
\tr \log f = \log \ddet_{c_0 (\Gamma)} f \quad \text{in} \; c_0 (\Gamma)
\end{equation}
for $f \in 1 + p M_r (\widehat{\Z\Gamma})$. Namely, for such $f$ we have $g = \det_{c_0 (\Gamma)} (f)$ and hence
\begin{align*}
\log_p \ddet_{\Gamma} f & \overset{\eqref{eq:17}}{:=} \log_p \ddet_{\Gamma} (\ddet_{c_0 (\Gamma)} f) \\
& = \tr_{\Gamma} \log \ddet_{c_0 (\Gamma)} f \\
& \overset{\eqref{eq:18}}{=} \tr_{\Gamma} \tr \log f \\
& = \tr_{\Gamma} \log f \overset{\eqref{eq:13}}{=:} \log_p \ddet_{\Gamma} f \; .
\end{align*}
One way to prove equation \eqref{eq:18} is by embedding the Tate algebra
\[
c_0 (\Gamma) \cong \Q_p \langle t^{\pm 1}_1 , \ldots , t^{\pm 1}_N \rangle \; ,
\]
into its quotient field and applying \cite[Appendix C, Lemma 4.1]{H}.

\begin{cor}
\label{t3.4}
For $\Gamma = \Z^N$ the limit formula in Theorem \ref{t3.2} extends to all $f \in \GL_r (c_0 (\Gamma))$.
\end{cor}

\begin{proof}
Using \cite[\S\,9, N$^o$ 4, Lemma 1]{Bourbaki} we find
\[
\ddet_{\Q_p} (f^{(n)}) = \ddet_{\Q_p} (\varepsilon^{(n)}) \quad \text{where} \; \varepsilon^{(n)} = \ddet_{c_0 (\Gamma^{(n)})} (f^{(n)}) \; .
\]
Equation \eqref{eq:16} gives
\[
\varepsilon^{(n)} = (\ddet_{c_0 (\Gamma)} (f))^{(n)} = (p^{\nu} \zeta \gamma)^{(n)} g^{(n)} \; .
\]
Since $\log_p$ vanishes on roots of unity and powers of $p$ we obtain
\[
\log_p \ddet_{\Q_p} (f^{(n)}) = \log_p \ddet_{\Q_p} (g^{(n)}) \; .
\]
On the other hand, by definition \eqref{eq:17} we have
\[
\log_p \ddet_{\Gamma} f = \log_p \ddet_{\Gamma} g \; .
\]
Now the Corollary follows from Theorem \ref{t3.2} applied to $g$. 
\end{proof}

\begin{remark}
\label{t3.5}
By Corollary \ref{t3.3} we know that $\log_p \det_{\Gamma} f = 0$ for $f$ in $\GL_r (\Z\Gamma)$ where $\Gamma = \Z^N$. This can also be seen from definition \eqref{eq:17} since $\det_{c_0 (\Gamma)} (f) \in (\Z\Gamma)^{\times} = \mu_2 \Gamma$, so that $g = 1$. Note here that $\Z\Gamma \cong \Z [t^{\pm 1}_1 , \ldots , t^{\pm 1}_N]$.
\end{remark}

We end this section with a reformulation of the existence of limits of renormalized $p$-adic logarithms. Consider the direct product decompositions
\[
\Q^{\times}_p = \mu_{p-1} \times (1 + p \Z_p) \times p^{\Z} \quad \text{for} \; p \neq 2
\]
and
\[
\Q^{\times}_2 = \mu_2 \times (1 + 4\Z_2) \times 2^{\Z} \quad \text{for} \; p = 2 \; .
\]
For $\chi \in \Q^{\times}_p$ let $\chi^{(1)}$ be the component in $1 + p \Z_p$ for $p \neq 2$ resp. in $1 + 4 \Z_2$ for $p = 2$. 

\begin{prop}
\label{t3.6}
Consider numbers $\chi_n \in \Q^{\times}_p , a_n \in \Z \setminus 0$ and $h_p \in \Q_p$. Then the following assertions are equivalent:
\begin{compactitem}
\item[1)] $\lim_{n\to \infty} \frac{1}{a_n} \log_p \chi_n = h_p$ in $\Q_p$.
\item[2)] For some (equivalently, all) integer(s) $k \in \Z \setminus 0$ with $|kh_p| < 1$ if $p \neq 2$ resp. $|kh_2| < 1/2$ if $p = 2$, there is a sequence $(u_n)$ in $\Q^{\times}_p$ with $u_n \to 1$ such that
\[
\chi^{(1) k}_n = \exp_p (kh_p)^{a_n} u^{ka_n}_n \quad \text{for large enough} \; n \; .
\]
\end{compactitem}
In particular
\[
\chi^{(1)k}_n \sim \exp_p (kh_p)^{a_n} \quad \text{in} \; \Q_p \; \text{as} \; n \to \infty \; .
\]
\end{prop}

\begin{proof}
1) $\Rightarrow$ 2) Writing
\[
\frac{1}{a_n} \log_p \chi_n = h_p + \varepsilon_n \quad \text{with} \; \varepsilon_n \to 0
\]
we get
\[
\log_p \chi^{(1) k}_n = k \log_p \chi^{(1)}_n = k \log_p \chi_n = a_n k h_p + a_n k \varepsilon_n \; .
\]
By assumption $\chi^{(1)k}_n$ lies in $1 + p \Z_p$ for $p \neq 2$ and in $1 + 4 \Z_2$ for $p = 2$. Hence we have
\[
\exp_p \log_p \chi^{(1)k}_n = \chi^{(1)k}_n \; .
\]
On the other hand $|kh_p| < ( \sqrt[p-1]{p})^{-1}$ by assumption and $|\varepsilon_n| < (\sqrt[p-1]{p})^{-1}$ for $n$ large enough. Hence we find
\[
\exp_p (a_n kh_p + a_n k\varepsilon_n) = \exp_p (kh_p)^{a_n} \exp_p (\varepsilon_n)^{ka_n} \; .
\]
Setting $u_n := \exp_p (\varepsilon_n)$ we obtain the formula for $\chi^{(1) k}_n$ claimed in 2).\\
2) $\Rightarrow$ 1) follows similarly by applying $\log_p$ in 2).
\end{proof}
\section{$p$-adic $R$-torsion} \label{sec:4}
In this section we prove that a renormalized $p$-adic limit of logarithms of multiplicative Euler characteristics exists and equals a suitably defined $p$-adic version of $R$-torsion.

Consider a discrete group $\Gamma$ and a (left) $\Z\Gamma$-module $M$ of type $FL$. Choose a resolution
\begin{equation}
\label{eq:18a}
0 \longrightarrow F_d \xrightarrow{\partial} \ldots \xrightarrow{\partial} F_0 \xrightarrow{\pi} M \longrightarrow 0
\end{equation}
of $M$ by finitely generated free $\Z \Gamma$-modules $F_{\nu}$ for $0 \le \nu \le d$. For each $\nu$ we also fix a $\Z\Gamma$-basis $\eb_{\nu}$ of $F_{\nu}$. Assume that the complex
\begin{equation}
\label{eq:19}
0 \longrightarrow \widehat{\Z\Gamma} \otimes_{\Z\Gamma} F_d \xrightarrow{\hat{\partial}} \ldots \xrightarrow{\hat{\partial}} \widehat{\Z\Gamma} \otimes_{\Z\Gamma} F_0 \longrightarrow 0
\end{equation}
where $\hat{\partial} := \id \otimes \partial$ is acyclic. We write $\tau (F_{\hullet} , \eb_{\hullet})$ for the torsion in $K_T (\widehat{\Z\Gamma})$ of this complex equipped with the induced bases $1 \otimes \eb_{\nu}$, cf. \cite[\S\,3,4]{Mi}. Consider an exact sequence of $\Z\Gamma$-modules
\[
0 \longrightarrow M' \longrightarrow M \longrightarrow M'' \longrightarrow 0
\]
of $\Z\Gamma$-modules. Assume that $(F'_{\hullet} \to M' , \eb'_{\hullet})$ and $(F''_{\hullet} \to M'' , \eb''_{\hullet})$ are based resolutions as above such that the complexes $\widehat{\Z\Gamma} \otimes_{\Z\Gamma} F'_{\hullet}$ and $\widehat{\Z\Gamma} \otimes_{\Z\Gamma} F''_{\hullet}$ are exact. Equipping an $FL$-resolution $F_{\hullet} \to M$ as in Proposition \ref{t1.8} for $R = \Z\Gamma$ with bases $\eb_{\nu} = \eb'_{\nu} \cup \eb''_{\nu}$ on $F_{\nu} = F'_{\nu} \oplus F''_{\nu}$ we have
\begin{equation}
\label{eq:19a}
\tau (F_{\hullet} , \eb_{\hullet}) = \tau (F'_{\hullet} , \eb'_{\hullet}) \tau (F''_{\hullet} , \eb''_{\hullet}) \quad \text{in} \; K_T (\widehat{\Z\Gamma}) \; .
\end{equation}
This is a special case of \cite[Theorem 3.1, \S\,4]{Mi}. According to \cite[\S\,15,16]{C} there is the following formula for the torsion: Since \eqref{eq:19} is an acyclic complex of free and hence projective $\widehat{\Z\Gamma}$-modules, it admits a chain contraction i.e. a degree-one $\widehat{\Z\Gamma}$-module homomorphism
\[
\delta : \widehat{\Z\Gamma} \otimes_{\Z\Gamma} F_{\hullet} \longrightarrow \widehat{\Z\Gamma} \otimes_{\Z\Gamma} F_{\hullet} 
\]
such that $\delta \hpartial + \hpartial \delta = \id$. Set
\[
F_{\odd} = F_1 \oplus F_3 \oplus \ldots \quad \text{and} \quad F_{\even} = F_0 \oplus F_2 \oplus \ldots
\]
and
\[
\eb_{\odd} = \eb_1 \cup \eb_3 \cup \ldots \quad \text{and} \quad \eb_{\even} = \eb_0 \cup \eb_2 \cup \ldots
\]
The map
\[
(\hpartial + \delta)_{\odd} := (\hpartial + \delta) \, |_{\widehat{\Z\Gamma} \otimes_{\Z\Gamma} F_{\odd}} : \widehat{\Z\Gamma} \otimes_{\Z\Gamma} F_{\odd} \longrightarrow \widehat{\Z\Gamma} \otimes_{\Z\Gamma} F_{\even}
\]
is an isomorphism of finitely generated free $\widehat{\Z\Gamma}$-modules. Using the chosen bases $1 \otimes \eb_{\odd}$ and $1 \otimes \eb_{\even}$ together with some ordering we may view $(\hpartial + \delta)_{\odd}$ as an automorphism $(\hpartial + \delta)^{\eb_{\hullet}}_{\odd}$ of $(\widehat{\Z\Gamma})^r$ where
\begin{equation}
\label{eq:20}
r = \sum_{i \, \odd} \rank_{\Z\Gamma} F_i = \sum_{i \, \even} \rank_{\Z\Gamma} F_i \; .
\end{equation}
We identify $\Aut_{\widehat{\Z\Gamma}} (\widehat{\Z\Gamma})^r$ with $\GL_r (\widehat{\Z\Gamma})$ by sending $f$ to the matrix $\langle f \rangle = (a_{ij})$ defined by $f (e_i) = \sum_j a_{ij} e_j$ where $e_1 , \ldots , e_r$ is the standard basis of $( \widehat{\Z\Gamma})^r$. Let
\[
\pi : \GL_r (\widehat{\Z\Gamma}) \longrightarrow K_T (\widehat{\Z\Gamma})
\]
be the canonical homomorphism. Then we have
\begin{equation}
\label{eq:21}
\tau (F_{\hullet} , \eb_{\hullet}) = \pi \langle (\hpartial + \delta)^{\eb_{\hullet}}_{\odd} \rangle \; .
\end{equation}
This is independent of the chosen orderings of the bases. For different bases $\eb'_{\nu}$ of $F_{\nu}$ it follows e.g. from \eqref{eq:21} that
\begin{equation}
\label{eq:22}
\tau (F_{\hullet} , \eb'_{\hullet}) = \tau (F_{\hullet} , \eb_{\hullet}) \pi (S) \quad \text{in} \; K_T (\widehat{\Z\Gamma})
\end{equation}
for some matrix $S \in \GL_r (\Z\Gamma)$.

Now assume that $\Gamma$ is countable and residually finite and that $\Wh^{\F_p} (\Gamma)$ is a torsion group. Then we can apply the homomorphism
\[
\log_p \ddet_{\Gamma} : K_T (\widehat{\Z\Gamma}) \longrightarrow \Q_p
\]
of Theorem \ref{t3.1} to $\tau (F_{\hullet} , \eb_{\hullet})$. We call
\begin{equation}
\label{eq:23}
\tau^{\Gamma}_p (M) := \log_p \ddet_{\Gamma} (\tau (F_{\hullet}, \eb_{\hullet})) \in \Q_p
\end{equation}
the $p$-adic $R$-torsion of $M$. Because of formula \eqref{eq:22} and Corollary \ref{t3.3} it is independent of the choice of bases $\eb_{\hullet}$. In the situation of \eqref{eq:19a} we have
\begin{equation}
\label{eq:23a} 
\tau^{\Gamma}_p (M) = \tau^{\Gamma}_p (M') + \tau^{\Gamma}_p (M'') \; .
\end{equation}

\begin{theorem}
\label{t4.1}
Let $\Gamma$ be countable and residually finite with $\Wh^{\F_p} (\Gamma)$ torsion. Let $M$ be a $\Z\Gamma$-module of type $FL$ with $pM = M$ and $\varprojlim_i M_{p^i} = 0$. Then the following assertions hold:\\
a) The $p$-adic Reidemeister torsion $\tau^{\Gamma}_p (M) \in \Q_p$ is defined and independent of the chosen based resolution \eqref{eq:18a} of $M$.\\
b) Choose a sequence $\Gamma_n \to e$ of normal subgroups $\Gamma_n$ of finite index in $\Gamma$ and consider the multiplicative Euler characteristics $\chi (\Gamma_n , M) \in \Q^{\times}$ from \eqref{eq:9a}. Then we have
\begin{equation}
\label{eq:24}
\tau^{\Gamma}_p (M) = \lim_{n \to \infty} (\Gamma : \Gamma_n)^{-1} \log_p \chi (\Gamma_n , M) \quad \text{in} \; \Q_p \; .
\end{equation}
\end{theorem}

\begin{example}
For $d = 1$ i.e. for resolutions of the form
\[
0 \longrightarrow (\Z\Gamma)^s \xrightarrow{\partial} (\Z \Gamma)^r \longrightarrow M \longrightarrow 0
\]
we showed in Proposition \ref{t2.2} that the conditions on $M$ in Theorem \ref{t4.1} are equivalent to $r = s$ and $\partial \in M_r (\Z\Gamma) \cap \GL_r (\widehat{\Z\Gamma})$. In this case we have $\tau^{\Gamma}_p (M) = \log_p \det_{\Gamma} \partial$ and $\chi (\Gamma_n , M) = |\det_{\Q_p} (\partial^{(n)})|$ using \eqref{eq:26} and an argument as in the proof of Proposition \ref{t4.2}. Theorem \ref{t4.1} reduces to the limit formula in Theorem \ref{t3.2}.
\end{example}

\begin{rem}
Given an exact sequence $0 \to M' \to M \to M'' \to 0$ of $\Z\Gamma$-modules, where $M'$ and $M''$ satisfy the conditions of Theorem \ref{t4.1}, the module $M$ satisfies those conditions as well by Proposition \ref{t1.8}. According to \eqref{eq:23a} we therefore have $\tau^{\Gamma}_p (M) = \tau^{\Gamma}_p (M') + \tau^{\Gamma}_p (M'')$. This also follows from equation \eqref{eq:24}.
\end{rem}

\begin{proof}
By Theorem \ref{t2.1} we know that for any resolution \eqref{eq:18a} of $M$ the complex \eqref{eq:19} is acyclic. Hence $\tau^{\Gamma}_p (M)$ is defined. Independence of the resolution will follow from b). Again by Theorem \ref{t2.1}, the homology groups $H_i (\Gamma_n , M)$ are finite for all $i$ and vanish for almost all $i$, so that
\[
\chi (\Gamma_n , M) = \prod_i |H_i (\Gamma_n , M) |^{(-1)^i} \in \Q^{\times}
\]
is well defined. In order to prove formula \eqref{eq:24} we apply Theorem \ref{t3.2} to $f = (\hpartial + \delta)^{\eb_{\hullet}}_{\odd} \in \GL_r (\widehat{\Z\Gamma})$. This gives the equality
\begin{equation}
\label{eq:25}
\tau^{\Gamma}_p (M) = \lim_{n\to \infty} (\Gamma : \Gamma_n)^{-1} \log_p (\ddet_{\Q_p} (f^{(n)})) \; .
\end{equation}
Let $(F^{(n)}_{\hullet} , \partial^{(n)})$ be the complex of free $\Z\Gamma^{(n)}$-modules of finite rank obtained by tensoring the complex $(F_{\hullet} , \partial)$ with $\Z\Gamma^{(n)}$ via the natural map $\Z\Gamma \to \Z \Gamma^{(n)}$. We have
\begin{equation}
\label{eq:26}
F^{(n)}_{\hullet} = \Z \Gamma^{(n)} \otimes_{\Z\Gamma} F_{\hullet} = (\Z \otimes_{\Z\Gamma_n} \Z\Gamma) \otimes_{\Z\Gamma} F_{\hullet} = \Z \otimes_{\Z\Gamma_n} F_{\hullet} \; .
\end{equation}
Here $\Z$  is viewed as a $\Z\Gamma_n$-module via the augmentation and the right $\Z\Gamma$-equivariant isomorphism
\[
\Z \otimes_{\Z\Gamma_n} \Z\Gamma \silo \Z \Gamma^{(n)}
\]
is induced by the map $\Z \times \Z\Gamma \to \Z \Gamma^{(n)}$ which sends $(\nu , \sum n_{\gamma} \gamma)$ to $\nu \sum_{\gamma} n_{\gamma} \overline{\gamma}$. We will use the identifications $F^{(n)}_{\hullet} = \Z \otimes_{\Z\Gamma_n} F_{\hullet} , \partial^{(n)} = \id \otimes \partial$ and view $(F^{(n)}_{\hullet} , \partial^{(n)})$ simply as a complex of finitely generated free $\Z$-modules. Tensoring the acyclic complex \eqref{eq:19} of $\widehat{\Z\Gamma}$-modules via the natural map
\[
\widehat{\Z\Gamma} = \varprojlim_i \Z\Gamma / p^i \Z \Gamma \longrightarrow \varprojlim_i \Z \Gamma^{(n)} / p^i \Z \Gamma^{(n)} = \Z_p \Gamma^{(n)}
\]
with $\Z_p \Gamma^{(n)}$ we obtain the complex of free $\Z_p$-modules
\begin{equation}
\label{eq:27}
\Z_p \Gamma^{(n)} \otimes_{\widehat{\Z\Gamma}} \widehat{\Z\Gamma} \otimes_{\Z\Gamma} F_{\hullet} = \Z_p \otimes_{\Z} \Z \Gamma^{(n)} \otimes_{\Z\Gamma} F_{\hullet} = \Z_p \otimes_{\Z} F^{(n)}_{\hullet} = \Z_p \otimes_{\Z\Gamma_n} F_{\hullet} \; .
\end{equation}
The differential $\hpartial^{(n)} := \id \otimes_{\widehat{\Z\Gamma}} \hpartial$ on the left corresponds to the differential $\id \otimes_{\Z\Gamma_n} \partial$ on $\Z_p \otimes_{\Z\Gamma_n} F_{\hullet}$. The complex \eqref{eq:27} carries the chain contraction $\delta^{(n)} := \id \otimes_{\widehat{\Z\Gamma}} \delta$ and it is therefore acyclic. The isomorphism
\[
(\hpartial + \delta)^{(n)}_{\odd} := \id \otimes_{\widehat{\Z\Gamma}} (\hpartial + \delta)_{\odd}
\]
can be identified with the isomorphism of $\Z_p$-modules
\begin{equation}
\label{eq:28a}
(\id \otimes_{\Z\Gamma_n} \partial + \delta^{(n)})_{\odd} := \id \otimes_{\Z\Gamma_n} \partial + \delta^{(n)} \, |_{\Z_p \otimes_{\Z\Gamma_n} F_{\odd}} : \Z_p \otimes_{\Z\Gamma_n} F_{\odd} \silo \Z_p \otimes_{\Z\Gamma_n} F_{\even} \; .
\end{equation}
Fixing $\Z$-bases of the free $\Z$-modules $\Z \otimes_{\Z\Gamma_n} F_{\hullet} = F^{(n)}_{\hullet}$ we may view \\
$(\id \otimes_{\Z\Gamma_n} \partial + \delta^{(n)})_{\odd}$ as an element of $\GL_{r (\Gamma : \Gamma_n)} (\Z_p)$ with $r$ defined in \eqref{eq:20} above. Its image in $K_T (\Q_p) = K_1 (\Q_p) / \langle \pm 1 \rangle$ is the torsion $\tau (\Q_p \otimes_{\Z\Gamma_n} F_{\hullet})$ of the based acyclic complex $(\Q_p \otimes_{\Z\Gamma_n} F_{\hullet} , \id \otimes_{\Z\Gamma_n} \partial)$. It is independent of the chosen $\Z$-bases since $K_1 (\Z) = \Z^{\times} = \{ \pm 1 \}$. In formula \eqref{eq:25} the matrix $f^{(n)}$ for $f = (\hpartial + \delta)^{\eb}_{\odd}$ is viewed as an element of $\GL_{r (\Gamma : \Gamma_n)} (\Z_p)$. The image of $f^{(n)}$ in $K_T (\Q_p)$ is equal to $\tau (\Q_p \otimes_{\Z\Gamma_n} F_{\hullet})$ by the above considerations. Consider the composition
\[
K_T (\Q_p) \overset{\det_{\Q_p}}{\silo} \Q^{\times}_p / \{ \pm 1 \} \; .
\]
Since $\det_{\Q_p} f^{(n)} = \det_{\Q_p} \tau (\Q_p \otimes_{\Z\Gamma_n} F_{\hullet})$ in \eqref{eq:25}, it suffices to show that
\begin{equation}
\label{eq:28}
\ddet_{\Q_p} \tau (\Q_p \otimes_{\Z\Gamma_n} F_{\hullet}) = \chi (\Gamma_n , M) \quad \text{in} \; \Q^{\times}_p / \{ \pm 1 \} \; .
\end{equation}
By Theorem \ref{t2.1} the homology groups $H_i (\Gamma_n , M) = H_i (\Z \otimes_{\Z\Gamma_n} F_{\hullet})$ are finite. Hence the complex $\Q \otimes_{\Z\Gamma_n} F_{\hullet}$ is acyclic. Its torsion
\[
\tau (\Q \otimes_{\Z\Gamma_n} F_{\hullet}) \in K_T (\Q) \overset{\det_{\Q}}{\silo} \Q^{\times} / \{ \pm 1 \}
\]
with respect to any $\Z$-bases of the groups $\Z \otimes_{\Z \Gamma_n} F_{\hullet}$ is independent of the choice of bases since $K_1 (\Z) = \{ \pm 1 \}$. Since $\tau (\Q \otimes_{\Z\Gamma_n} F_{\hullet})$ is mapped to $\tau (\Q_p \otimes_{\Z\Gamma_n} F_{\hullet})$ under the natural map $K_T (\Q) \to K_T (\Q_p)$, we have
\[
\ddet_{\Q_p} \tau (\Q_p \otimes_{\Z\Gamma_n} F_{\hullet}) = \ddet_{\Q} \tau (\Q \otimes_{\Z\Gamma_n} F_{\hullet}) \quad \text{in} \; \Q^{\times} / \{ \pm 1 \} \; .
\]
Thus \eqref{eq:28} is a consequence of the formula
\[
\ddet_{\Q} \tau (\Q \otimes_{\Z\Gamma_n} F_{\hullet}) = \chi (\Gamma_n , M) \quad \text{in} \; \Q^{\times} / \{ \pm 1 \}
\]
which is a special case of the following proposition applied to $L_{\hullet} = \Z \otimes_{\Z\Gamma_n} F_{\hullet}$ noting that $H_i (L_{\hullet}) = H_i (\Gamma_n , M)$. 
\end{proof}

\begin{prop}
\label{t4.2}
Let $0 \to L_d \xrightarrow{\partial} \ldots \xrightarrow{\partial} L_0 \to 0$ be a complex of finitely generated free $\Z$-modules whose homology groups are finite. The torsion \\
$\tau (\Q \otimes L_{\hullet}) \in K_T (\Q)$ of the acyclic complex $\Q \otimes L_{\hullet}$ equipped with bases coming from $\Z$-bases of the $L_{\nu}$ is independent of the bases and we have
\[
\ddet_{\Q} \tau (\Q \otimes L_{\hullet}) = \prod_i |H^i (L_{\hullet})|^{(-1)^i} \quad \text{in} \; \Q^{\times} / \{ \pm 1 \} \; .
\]
\end{prop}

\begin{proof}
By induction. For $d = 0$ the assertion is trivial. For $d = 1$ we have $0 \to L_1 \xrightarrow{\partial} L_0 \to 0$ with $\partial_{\Q} : \Q \otimes L_1 \silo \Q \otimes L_0$ an isomorphism. Then $\delta = \partial^{-1}_{\Q}$ gives a chain retraction and
\[
\ddet_{\Q} \tau (\Q \otimes L_{\hullet}) = \ddet_{\Q} ((\partial_{\Q} + \delta)_{\odd}) = \ddet_{\Q} (\partial_{\Q}) \in \Q^{\times} / \{ \pm 1 \} \; .
\]
Here, after identifying $L_0$ and $L_1$ with $\Z^N$ for $N = \rk L_0 = \rk L_1$ we view $\partial_{\Q}$ as an automorphism of $\Q^N$. The map $\partial$ is injective since $\partial_{\Q}$ is injective and since $L_1$ is free. It is known that 
\[
|\ddet (\partial_{\Q})| = |\coker \partial| \; ,
\]
as follows from the theorem on elementary divisors. Since $H_1 (L_{\hullet}) = \ker \partial = 0$ and $H_0 (L_{\hullet}) = \coker \partial$ it follows that
\[
\ddet_{\Q} \tau (\Q \otimes L_{\hullet}) = |H_0 (L_{\hullet})| \, |H_1 (L_{\hullet})|^{-1} \quad \text{in} \; \Q^{\times} / \{ \pm 1 \} \; .
\]
Now assume that $d \ge 2$ and that the proposition holds for $d$ replaced by $d-1$. Consider the commutative diagram with exact lines:
\[
\xymatrix{
 & 0 \ar[d] & 0 \ar[d] & 0 \ar[d] & \\
0 \ar[r] & L_d \ar@{=}[r] \ar[d] & L_d \ar[r] \ar[d] & 0 \ar[r] \ar[d] & 0 \\
 & \vdots \ar[d] & \vdots \ar[d] & \vdots \ar[d] & \\
0 \ar[r] & L_2 \ar@{=}[r] \ar[d] & L_2 \ar[r] \ar[d] & 0 \ar[r] \ar[d] &  0 \\
0 \ar[r] & \Ker \partial \ar[r] \ar[d] & L_1 \ar[r] \ar[d]_{\partial} & L_1 / \Ker \partial \ar[r] \ar[d]_{\partial} & 0 \\
0 \ar[r] & 0 \ar[r] \ar[d] & L_0 \ar[r] \ar[d] & L_0 \ar[r] \ar[d] & 0 \\
 & 0 & 0 & 0 &
}
\]
It defines an exact sequence of complexes
\[
0 \longrightarrow L'_{\hullet} \longrightarrow L_{\hullet} \longrightarrow L''_{\hullet} \longrightarrow 0 \; .
\]
Noting that $\Ker \partial \subset L_1$ and $L_1 / \Ker \partial \overset{\partial}{\hookrightarrow} L_0$ are injective we see that both $L'_{\hullet}$ and $L''_{\hullet}$ are complexes of finitely generated free $\Z$-modules. By \cite[Theorem 3.1]{Mi} we have
\begin{equation}
\label{eq:29} 
\tau (\Q \otimes L_{\hullet}) = \tau (\Q \otimes L'_{\hullet}) \tau (\Q \otimes L'') \quad \text{in} \; K_T (\Q)\; .
\end{equation}
It follows from the definition of torsion in \cite[\S\,3,4]{Mi} that shifting the numbering of a complex by one replaces its torsion by its inverse. Let $L'_{\hullet} [1]$ be the complex obtained from $L'_{\hullet}$ by putting $L_{\nu}$ in degree $\nu-1$ for $\nu \ge 2$ and $\Ker \partial$ in degree zero. Then we have
\begin{equation}
\label{eq:30}
\tau (\Q \otimes L'_{\hullet}) = \tau (\Q \otimes L'_{\hullet} [1])^{-1} \quad \text{in} \; K_T (\Q) \; ,
\end{equation}
and by the induction hypotheses
\[
\ddet_{\Q} \tau (\Q \otimes L'_{\hullet} [1]) = \prod^d_{i=1} |H_i (L_{\hullet})|^{(-1)^{i-1}} \quad \text{in} \; \Q^{\times} / \{ \pm 1 \} \; .
\]
Together with the equation from the case $d = 1$,
\[
\ddet_{\Q} \tau (\Q \otimes L''_{\hullet}) = |H_0 (L_{\hullet})|
\]
and \eqref{eq:29}, \eqref{eq:30} we get the assertion for $L_{\hullet}$:
\[
\ddet_{\Q} \tau (\Q \otimes L_{\hullet}) = \prod^d _{i=0} |H_i (L_{\hullet})|^{(-1)^i} \quad \text{in} \; \Q^{\times} / \{ \pm 1 \} \; .
\]
\end{proof}

In Theorem \ref{t4.1} we had to assume that $pM = M$ i.e. that $M$ was $p$-adically expansive of exponent zero because except for $\Gamma = \Z^N$ we only know how to define $\log_p \det_{\Gamma}$ on $K_T (\widehat{\Z\Gamma})$ and not on $K_T (c_0 (\Gamma))$.

Assume that for a $\Z\Gamma$-module of type $FL$ with a based resolution \eqref{eq:18a} the complex
\[
0 \longrightarrow c_0 (\Gamma) \otimes_{\Z\Gamma} F_d \longrightarrow \ldots \longrightarrow c_0 (\Gamma) \otimes_{\Z\Gamma} F_0 \longrightarrow 0
\]
is acyclic. Let $\tau (F_{\hullet} , \eb_{\hullet}) \in K_T (c_0 (\Gamma))$ be its torsion. It is multiplicative in short exact sequences as in \eqref{eq:19a}. Assume that we are given a ``reasonable'' definition of a homomorphism
\[
\log_p \ddet_{\Gamma} : K_T (c_0 (\Gamma)) \longrightarrow \Q_p \; .
\]
Then as in \eqref{eq:23} we can define the $p$-adic $R$-torsion
\[
\tau^{\Gamma}_p (M) = \log_p \ddet_{\Gamma} (\tau (F_{\hullet} , \eb_{\hullet})) \in \Q_p \; .
\]
It is independent of the choices of bases $\eb_{\nu}$ if we have $\log_p \det_{\Gamma} f = 0$ for $f \in \GL_r (\Z\Gamma)$. If $\Gamma$ is countable and residually finite we define ``reasonable'' to mean that $\log \det_{\Gamma}$ is compatible with \eqref{eq:13} and that the analogue of the limit formula
\[
\log_p \ddet_{\Gamma} f = \lim_{n\to\infty} (\Gamma : \Gamma)^{-1} \log_p (\ddet_{\Q_p} f^{(n)}) \quad \text{in} \; \Q_p
\]
should hold for all $f \in \GL_r (c_0 (\Gamma))$. The limit formula implies the above condition $\log_p \ddet_{\Gamma} f = 0$ for $f \in \GL_r (\Z \Gamma)$. Trivial modifications of the proof of Theorem \ref{t4.1} then give the following result which applies in particular to $\Gamma = \Z^N$.

\begin{theorem}
\label{t4.3}
Let $\Gamma$ be a countable and residually finite group for which there is a ``reasonable'' definition of $\log_p \det_{\Gamma}$ on $K_T (c_0 (\Gamma))$ in the above sense. Let $M$ be a $\Z \Gamma$-module of type $FL$ which is $p$-adically expansive and such that $\varprojlim_i M_{p^i} = 0$. Then $\tau^{\Gamma}_p (M) \in \Q_p$ is defined and independent of the resolution \eqref{eq:18a}. Moreover for any sequence $\Gamma_n \to e$ we have
\[
\tau^{\Gamma}_p (M) = \lim_{n\to \infty} (\Gamma : \Gamma_n)^{-1} \log_p \chi (\Gamma_n , M) \quad \text{in} \; \Q_p \; .
\]
\end{theorem}

Note that the multiplicative Euler characteristics $\chi (\Gamma_n , M)$ exist by Theorem \ref{t2.1}. It follows from the definition or the limit formula that $\tau^{\Gamma}_p$ is additive in short exact sequences $0 \to M' \to M \to M'' \to 0$. If the assumptions of the Theorem are satisfied for $M'$ and $M''$ they hold for $M$ as well by Proposition \ref{t1.8} for $R = \Z \Gamma$. As mentioned before, Theorem \ref{t4.3} applies to $\Gamma = \Z^N$ with $\log_p \det_{\Gamma}$ defined by formula \eqref{eq:17}, because in Corollary \ref{t3.4} we proved the required limit formula. In this case the conditions on $M$ in Theorem \ref{t4.3} can be simplified:

\begin{remark}
\label{t4.4a}
For $\Gamma = \Z^N$ a $\Z\Gamma$-module $M$ is of type $FL$ if and only if it is finitely generated. In this case there exist $FL$-resolutions \eqref{eq:18a} of length $d \le N+1$. Moreover we have $\varprojlim_i M_{p^i} = 0$.
\end{remark}

\begin{proof}
Since $\Z\Gamma \cong \Z [t^{\pm 1}_1 , \ldots , t^{\pm 1}_N]$ is a Noetherian ring every finitely generated module $M$ has a resolution by finitely generated free $\Z\Gamma$-modules
\[
\ldots \longrightarrow F_1 \xrightarrow{\partial} F_0 \longrightarrow M \longrightarrow 0 \; .
\]
The global dimension of $\Z\Gamma$ is $N+1$. Hence the projective dimension $d$ of $M$ is $d \le N+1$. By \cite[Lemma 4.1.6]{W}, in the resolution
\[
0 \longrightarrow F_d / \Ker \partial \xrightarrow{\partial} F_{d-1} \longrightarrow \ldots \longrightarrow F_0 \longrightarrow M \longrightarrow 0
\]
the finitely generated module $F_d / \Ker \partial$ is projective. It was pointed out in \cite[p. 111]{Sw} that the Quillen--Suslin theorem is valid for $\Z\Gamma$. Hence every projective module over $\Z\Gamma$ and in particular $F_d / \Ker \partial$ is free. The converse implication is clear. The vanishing of $\varprojlim_i M_{p^i}$ is a special case of Proposition \ref{t1.6}.
\end{proof}

\begin{cor}
\label{t4.4}
For $\Gamma = \Z^N$ let $M$ be a $p$-adically expansive $\Z\Gamma$-module. Then the $p$-adic $R$-torsion $\tau^{\Gamma}_p (M) \in \Q_p$ (defined using $\log_p \det_{\Gamma}$ from \eqref{eq:17}) satisfies
\[
\tau^{\Gamma}_p (M) = \lim_{n\to \infty} (\Gamma : \Gamma_n)^{-1} \log_p \chi (\Gamma_n , M) \quad \text{in} \; \Q_p
\]
for a sequence $\Gamma_n \to 0$.
\end{cor}

In section \ref{sec:6} we will see that for $\Gamma = \Z^N$ the multiplicative Euler characteristics $\chi (\Gamma_n , M)$ have a very nice interpretation as intersection numbers on arithmetic schemes. We end this section by mentioning a geometric characterization of $p$-adic expansiveness in the case $\Gamma = \Z^N$. Set $R = \Z\Gamma = \Z [t^{\pm 1}_1 , \ldots , t^{\pm 1}_N]$. Let
\[
T^N_p = \{ (z_1 , \ldots , z_n) \in \oQ^N_p \mid |z_i|_p = 1 \; \text{for} \; i = 1 , \ldots , N \}
\]
be the $p$-adic $N$-torus in $\oQ^N_p$. Recall that the support of an $R$-module $M$ is $\supp M = \spec R / \Ann (M)$ where $\Ann (M)$ is the annihilator ideal of $M$ in $R$. A prime ideal $\ep$ of $R$ is in $\supp M$ if and only if $M_{\ep} \neq 0$. For a finite $R$-module $M$ let $\ep_1 ,\ldots , \ep_r$ be the minimal prime ideals of $\supp M$. Then
\[
\supp M = V (\ep_1) \cup \ldots \cup V (\ep_r)
\]
with $V (\ep) = \spec R / \ep$ is the decomposition of $\supp M$ into irreducible components. The prime ideals $\ep_1 , \ldots , \ep_r$ are also the minimal associated prime ideals of $M$, see \cite[Theorem 6.5]{M}. In \cite[Theorem 3.3]{B} the following characterizations of $p$-adic expansiveness are shown:

\begin{theorem}[Br\"auer]
\label{t4.5}
For $\Gamma = \Z^N$ a finitely generated $\Z\Gamma$-module $M$ is $p$-adically expansive if and only if either of the following equivalent conditions holds:\\
a) 
\begin{equation}
\label{eq:31}
T^N_p \cap (\supp M) (\oQ_p) = \emptyset \; .
\end{equation}
b) The module $M$ is $S_p$-torsion, where $S_p$ is the multiplicative system $S_p = \Z\Gamma \cap c_0 (\Gamma)^{\times}$. 
\end{theorem}

\cite{B} formulates condition \eqref{eq:31} in terms of associated primes of $M$. This is equivalent to our formulation since $\Ass (M)$ and $\supp M$ have the same minimal primes. 
\section{Algebraic dynamical systems} \label{sec:5}
Actions of countable discrete groups $\Gamma$ on compact abelian topological groups $X$ by continuous group automorphisms are called ``algebraic''. The Pontrjagin dual $M = \aX$ is a discrete abelian group with the induced $\Gamma$-action, and hence a $\Z\Gamma$-module. In this way we obtain a contravariant equivalence of the category of algebraic $\Gamma$-actions and the category of $\Z\Gamma$-modules $M$. For commutative groups $\Gamma$ the latter category is equivalent to the category of quasicoherent sheaves on the scheme $\spec \Z\Gamma$. Algebraic dynamical systems are studied in depth in the book \cite{Sch} with an emphasis on the case $\Gamma = \Z^N$. Often dynamical properties of the $\Gamma$-action on $X$ have an appealing algebraic or (for $\Gamma = \Z^N$) algebraic-geometric characterization in terms of the $\Z\Gamma$-module $M$ or the corresponding sheaf on $\A^N = \spec \Z \Gamma$. For more studies of algebraic actions of non-commutative groups we refer e.g. to \cite{D1}, \cite{CL}, \cite{LT} for example. 

An algebraic action of $\Gamma$ on $X$ is expansive if and only if there is a neighborhood $U$ of $0 \in X$ such that $\bigcap_{\gamma \in \Gamma} \gamma U = \{ 0 \}$. Expansiveness is equivalent to the condition that the $\Z\Gamma$-module $M$ is finitely generated and $L^1 \Gamma \otimes_{\Z\Gamma} M = 0$ cf. \cite[Theorem 3.1]{CL}. For $\Gamma = \Z^N$ it is also equivalent to the condition
\begin{equation}
\label{eq:32}
T^N \cap (\supp M) (\C) = \emptyset \; .
\end{equation}
Here
\[
T^N = \{ (z_1 , \ldots , z_N) \in \C^N \mid |z_i| = 1 \; \text{for} \; i = 1 , \ldots , N \}
\]
is the real $N$-torus in $\C^N$. This follows from \cite[Theorem 6.5]{Sch} because the minimal primes of $\Ass (M)$ and $\supp M$ are the same. For discrete groups $\Gamma$ there is a $p$-adic analogue of $L^1 \Gamma$, however in the $p$-adic context, because of the ultrametric inequality it is more natural to work with the algebra $c_0 (\Gamma)$. Thus it is natural to call the $\Gamma$-action $p$-adically expansive if $c_0 (\Gamma) \otimes_{\Z\Gamma} M = 0$ and --- for technical reasons --- if the $\Z\Gamma$-module $M$ is of type $FP$. In other words, the $\Gamma$-action on $X$ is $p$-adically expansive if and only if $M$ is a $p$-adically expansive $\Z\Gamma$-module. Incidentally, for $\Gamma = \Z^N$ a $\Z\Gamma$-module is of type $FP$ if and only if it is finitely generated. The analogy between conditions \eqref{eq:31} and \eqref{eq:32} in the classical and the $p$-adic case lends further credibility to our definition of $p$-adic expansiveness.

\begin{prop}
\label{t5.1}
Consider an algebraic action of a countable discrete group $\Gamma$ on a compact abelian group $X$ such that $M = \aX$ is a $\Z\Gamma$-module of type $FL$. Then the following conditions are equivalent:\\
a) The action of $\Gamma$ on $X$ is $p$-adically expansive.\\
b) There is an integer $n_0 \ge 0$ with $X_{p^n} = X_{p^{n_0}}$ for all $n \ge n_0$ where $X_{p^n} = \ker (p^n : X \to X)$.
\end{prop}

\begin{proof}
We know that a) is equivalent to $M / p^n M = M / p^{n_0} M$ for some $n_0$ and all $n \ge n_0$. Via Pontrjagin duality this is equivalent to b). 
\end{proof}
It is somewhat odd that condition b) depends only the structure of $X$ as an abelian group and not as in the classical case on the $\Gamma$-action.

For algebraic actions of amenable groups $\Gamma$ the topological entropy $h$ coincides with the measure theoretic entropy with respect to the Haar probability measure on $X$. Both definitions of entropy do not seem to have an analogue with values in the $p$-adic numbers. For an expansive action of $\Gamma$ on $X$ it follows from the definition that for every cofinite normal subgroup $\Delta$ of $\Gamma$ the number of fixed points $X^{\Delta}$ of $\Delta$ on $X$ is finite. If $\Gamma$ is residually finite one may therefore consider the periodic points entropy
\begin{equation}
\label{eq:33}
h^{\per} := \lim_{n\to \infty} (\Gamma : \Gamma_n)^{-1} \log |X^{\Gamma_n}| \in [0,\infty] \quad \text{for} \; \Gamma_n \to e
\end{equation}
if it exists. For $f \in M_N (\Z\Gamma)$ the $\Gamma$-action on $X_f = \aM_f$ with $M_f = (\Z\Gamma)^N / (\Z\Gamma)^N f$ is expansive if and only if $L^1 \Gamma \otimes_{\Z\Gamma} M_f = 0$ i.e. if there is some $g \in M_N (L^1 \Gamma)$ with $gf = 1$. Since as mentioned above $M_N (L^1 \Gamma)$ is directly finite this means that $f \in \GL_N (L^1 \Gamma)$. In this case, at least for $N = 1$ it is known by \cite{DS} that $h^{\per}$ exists and that it can be expressed by the Fuglede--Kadison determinant on $K_1$ of the von Neumann algebra of $\Gamma$
\begin{equation}
\label{eq:34}
h^{\per} = \log \ddet_{\Nh\Gamma} f \; .
\end{equation}
Moreover, if in addition $\Gamma$ is amenable so that the entropy of the $\Gamma$-action on $X_f$ is defined, we have
\begin{equation}
\label{eq:35}
h = h^{\per} \; .
\end{equation}
In \cite{D} we replaced the classical logarithm in formula \eqref{eq:33} by the $p$-adic logarithm $\log_p : \Q^{\times}_p \to \Z_p$ and called the resulting quantity
\begin{equation}
\label{eq:36}
h^{\per}_p := \lim_{n\to \infty} (\Gamma : \Gamma_n)^{-1} \log_p |X^{\Gamma_n}|
\end{equation}
the $p$-adic (periodic points) entropy if the limit exists for all $\Gamma_n \to e$. We also introduced $\log_p \ddet_{\Gamma} f$ as an analogue of $\log \det_{\Nh\Gamma} f$. For $X = X_f$ we proved the formula
\begin{equation}
\label{eq:37}
h^{\per}_p = \log_p \ddet_{\Gamma} f
\end{equation}
under the additional conditions on $\Gamma$ and $f$ of Theorem \ref{t3.2}. For $\Gamma = \Z^N$ formula \eqref{eq:35} is known to be true for arbitrary expansive algebraic actions by \cite[Theorem 21.1]{Sch}. 

In the $p$-adic case however, it was noted by Br\"auer \cite[Example 7.1]{B}, that even for $\Gamma = \Z$ there exist $p$-adically expansive actions for which the correspoding limit \eqref{eq:36} does not exist: Consider
\[
\F_4 = \F_2 [t] / (t^2 + t + 1) = \Z [t , t^{-1}] / (2 , t^2 + t + 1)
\]
as a $\Z [\Z] = \Z [t , t^{-1}]$-module. Thus $t$ acts by multiplication with the generator $\xi := t \mod (2 , t^2 + t + 1)$ of $\F^{\times}_4$ on $\F_4$. Fix a prime number $p \neq 2$. Then multiplication by $p$ on $M = \F_4$ is invertible. Hence $M$ is $p$-adically expansive (of exponent zero) and $\varprojlim_i M_{p^i} = 0$. Consider the sequence $\Gamma_n = 3n \Z \to 0$. We have $|X^{\Gamma_n}| = |X| = 4$ and hence
\[
(\Gamma : \Gamma_n)^{-1} \log_p |X^{\Gamma_n}| = (3n)^{-1} \log_p |X^{\Gamma_n}| = (3n)^{-1} \log_p 4 \; .
\]
Since $\log_p 4 \neq 0$ for $p \neq 2$ this sequence does not converge in $\Q_p$ for $n \to \infty$. Br\"auer showed that $\log_p \det_{\Z^N} f$ only depends on the $\Z [\Z^N]$-module $M_f$ and gave a natural extension of $\log_p \det_{\Z^N}$ to all $p$-adically expansive $\Z [\Z^N]$-modules $M$, see \eqref{eq:44} below. For the module $M = \F_4$ above he obtained the value $\log_p \det_{\Z^N} M = 0$, so that we should have $h^{\per}_p = 0$ in this case. The starting points of the present note were the following observations:
\begin{compactitem}
\setlength{\itemsep}{0pt}
\item For residually finite $\Gamma$ and $X = X_f$ with $f \in M_r (\Z \Gamma)$ $p$-adically expansive, we have $H^i (\Gamma_n ,X) = 0$ for $i \ge 1$ and therefore
\begin{equation}
\label{eq:38}
|X^{\Gamma_n}| = \chi (\Gamma_n , X) := \prod_i |H^i (\Gamma_n, X)|^{(-1)^i} \; .
\end{equation}
This is shown after the proof of Proposition \ref{t5.2} below.
\item For $\Gamma \cong \Z$ and any finite $\Gamma$-module $X$ we have $\chi (\Gamma , X) = 1$. This follows from the exact sequence
\[
0 \longrightarrow H^0 (\Gamma , X) \longrightarrow X \xrightarrow{f-\id} X \longrightarrow H^1 (\Gamma, X) \longrightarrow 0
\]
where $f$ is the automorphism of $X$ corresponding to a generator of $\Gamma \cong \Z$. 
\end{compactitem}
Thus, redefining the $p$-adic periodic points entropy as the limit
\begin{equation}
\label{eq:40}
h^{\per}_p := \lim_{n\to\infty} (\Gamma : \Gamma_n)^{-1} \log_p \chi (\Gamma_n , X) \quad \text{for} \; \Gamma_n \to e 
\end{equation}
if it exists, we get the same quantity as before for \textit{principal} $p$-adically expansive actions. Moreover, for the non-principal $\Gamma = \Z$-action on $X = \aF_4$ above we have $\chi (\Gamma_n , X) = 1$ for any non-trivial subgroup $\Gamma_n$ of $\Gamma$ and hence $h^{\per}_p = 0$ for all $\Gamma_n \to 0$ as suggested by Br\"auer's considerations.

Even in the classical case it looks natural in relation to entropy to consider the following limit --- if it exists
\begin{equation}
\label{eq:41}
\lim_{n\to \infty} (\Gamma : \Gamma_n)^{-1} \log \chi (\Gamma_n , X) \; .
\end{equation}
Namely, entropy is known to be additive in short exact sequences and the logarithmic Euler characteristic has the same property. Of course, with \eqref{eq:41} there is the additional complication that the groups $H^i (\Gamma_n , X)$ have to be finite and zero for large $i$. This is the case for expansive algebraic actions on $X$ if $M = \aX$ is of type $FL$ and if $L^1 \Gamma$ is a flat $\Z\Gamma$-module. However flatness fails if $\Gamma$ contains a free group in two variables, cf. \cite[Theorem 3.1.(4) and Remark 3.4]{CL}. On the other hand, by Theorem \ref{t2.1} and Proposition \ref{t5.2} below we know that $p$-adic expansiveness  and the condition $\varprojlim_i M_{p^i} = 0$ imply the finiteness of all $H^i (\Gamma_n , X)$.  We will now rephrase the main theorem of the previous section as a calculation of the modified $p$-adic periodic points entropy \eqref{eq:40}.

\begin{prop}
\label{t5.2}
Let $\Gamma$ be a discrete group and $M$ a discrete $\Z\Gamma$-module. Then the cohomology groups $H^i (\Gamma, X)$ of the compact Pontrjagin dual $X = \aM$ are compact and there is a natural topological isomorphism
\begin{equation}
\label{eq:41a}
H_i (\Gamma , M)^{\ast} = H^i (\Gamma , X) \quad \text{for each} \;  i \ge 0 \; .
\end{equation}
In particular $H^i (\Gamma, X)$ is finite (resp. zero) if and only if $H_i (\Gamma, M)$ is finite (resp. zero) and in this case $|H^i (\Gamma, X)| = |H_i (\Gamma, M)|$. Thus
\[
\chi (\Gamma, X) = \prod_i |H^i (\Gamma,X)|^{(-1)^i}
\]
is defined if and only
\[
\chi (\Gamma , M) = \prod_i |H_i (\Gamma , M)|^{(-1)^i}
\]
is defined. In this case we have $\chi (\Gamma,X) = \chi (\Gamma, M)$.
\end{prop}

\begin{proof}
For a free right $\Z\Gamma$-module $L$ on a set of generators $S$ the isomorphism
\[
\Hom_{\Z\Gamma} (L,X) = X^S \; , \; \alpha \longmapsto (\alpha (s))_{s \in S}
\]
turns $\Hom_{\Z\Gamma} (L,X)$ into a compact abelian group. The topology is independent of the choice of generators $S$ of $L$. There is a natural isomorphism of compact groups
\begin{equation}
\label{eq:42}
(L \otimes_{\Z\Gamma} M)^{\ast} \silo \Hom_{\Z\Gamma} (L,X) \; .
\end{equation}
It is obtained by sending a bilinear map $\varphi : L \times M \to S^1$ with $\varphi (la , m) = \varphi (l, am)$ for $a \in \Z\Gamma , l \in L , m \in M$ to the $\Z\Gamma$-equivariant map $\Phi : L \to \Hom (M, S^1)$ given by $\Phi (l) (m) = \varphi (l,m)$. Now choose a resolution $\ldots \to L_1 \to L_0 \to \Z \to 0$ of the $\Z\Gamma$-module $\Z$ by free right $\Z\Gamma$-modules $L_i$. Using \eqref{eq:42} we obtain a topological isomorphism of complexes of compact groups
\[
(L_{\hullet} \otimes_{\Z\Gamma} M)^{\ast} \silo \Hom_{\Z\Gamma} (L_{\hullet} , X) \; .
\]
Taking cohomology and using that Pontrjagin duality is an exact functor we obtain the desired topological isomorphism
\[
H_i (\Gamma , M)^{\ast} \silo H^i (\Gamma, X) \; .
\]
\end{proof}

In the situation of formula \eqref{eq:38} we have $c_0 (\Gamma) \otimes_{\Z \Gamma} M_f = 0$ and hence $c_0 (\Gamma)^r = c_0 (\Gamma)^r f$. Thus $f$ has a left inverse and using Theorem \ref{t2.3} it follows that $f \in \GL_r (c_0 (\Gamma))$. We get a short exact sequence
\[
0 \longrightarrow (\Z\Gamma)^r \xrightarrow{f} (\Z\Gamma)^r \longrightarrow M \longrightarrow 0 \; .
\]
Using the isomorphism $\Z \otimes_{\Z\Gamma_n} \Z\Gamma = \Z \Gamma^{(n)}$ we find
\[
H_i (\Gamma_n , M) = H_i ((\Z \Gamma^{(n)})^r \xrightarrow{f^{(n)}} (\Z\Gamma^{(n)})^r)
\]
and therefore $H_i (\Gamma_n , M) = 0$ for $i\ge 2$. Since $f \in \GL_r (c_0 (\Gamma))$ we have $f^{(n)} \in \GL_r (\Q_p \Gamma^{(n)})$. Thus the map $f^{(n)} : \Z \Gamma^{(n)} \to \Z \Gamma^{(n)}$ is injective and hence $H_1 (\Gamma_n , M) = 0$. Using Proposition \ref{t5.2} we conclude that in the situation of \eqref{eq:38} we have $H^i (\Gamma_n ,X) = 0$ for $i \ge 1$ as claimed in \eqref{eq:38}.

\begin{theorem}
\label{t5.3}
Under the conditions on $\Gamma$ and $M$ in Theorem \ref{t4.1} resp. Corollary \ref{t4.4} the $p$-adic entropy \eqref{eq:40} of $X = \aM$
\[
h^{\per}_p := \lim_{n\to\infty} (\Gamma : \Gamma_n)^{-1} \log_p \chi (\Gamma_n , X)\quad \text{for} \; \Gamma_n \to e
\]
is well defined and equal to the $p$-adic $R$-torsion of $M$:
\[
h^{\per}_p = \tau^{\Gamma}_p (M) \; .
\]
\end{theorem}

This is a reformulation of Theorem \ref{t4.1} and Corollary \ref{t4.4} taking into account Proposition \ref{t5.2}. It was motivated by an analogous (and deeper) result of Li and Thom \cite[Theorem 1.1]{LT}, who express classical entropy of algebraic dynamical systems of type $FL$ for amenable groups as an $L^2$-torsion. 

In the rest of this section we give a review of Br\"auer's definition of $\log_p \det_{\Gamma} (M)$ for $\Gamma = \Z^N$ and show that it agrees with $\tau^{\Gamma}_p (M)$.

Set $R = \Z [\Z^N]$ and let $\Mh_{S_p} (R)$ be the category of finitely generated $R$-modules which are $S_p = R \cap c_0 (\Z^N)^{\times}$-torsion. According to Theorem \ref{t4.5} these are exactly the $p$-adically expansive $R$-modules. This is an exact subcategory of the category of all $R$-modules and one has a localization sequence \cite[IX Theorem 6.3 and Corollary 6.4]{Ba}
\begin{equation}
\label{eq:42a}
R^{\times} = K_1 (R) \longrightarrow K_1 (R [S^{-1}_p]) \xrightarrow{\delta} K_0 (\Mh_{S_p} (R)) \longrightarrow K_0 (R) \xrightarrow{i} K_0 (R [S^{-1}_p]) \longrightarrow 0 \; .
\end{equation}
The map $i$ is injective since $K_0 (R) = K_0 (\Z) = \Z$ and since $K_0 (R [S^{-1}_p])$ surjects onto $\Z$ via the rank. Hence $\delta$ induces an isomorphism
\[
\overline{\delta} : K_1 (R [S^{-1}_p]) / R^{\times} \silo K_0 (\Mh_{S_p} (R)) \; ,
\]
whose inverse is denoted by $\cl_p$. Noting that $R^{\times} = \pm \Z^N$, we have:
\begin{equation}
\label{eq:43}
K_1 (c_0 (\Z^N)) / R^{\times} = K_T (c_0 (\Z^N)) \; .
\end{equation}
Following \cite[4.3]{B}, we may therefore consider the composition
\begin{equation}
\label{eq:44}
\log_p \ddet_{\Z^N} : K_0 (\Mh_{S_p} (R)) \overset{\cl_p}{\silo} K_1 (R [S^{-1}_p])/ R^{\times} \longrightarrow K_T (c_0 (\Z^N)) \xrightarrow{\log_p \ddet_{\Z^N}} \Q_p \; .
\end{equation}
Here the last map is defined by equation \eqref{eq:17}.

\begin{prop}
\label{t5.4}
Set $\Gamma = \Z^N$. For any $p$-adically expansive $R = \Z\Gamma$-module $M$ we have:
\[
\log_p \ddet_{\Gamma} [M] = \tau^{\Gamma}_p (M) \; .
\]
\end{prop}

\begin{proof}
We recall the definition of $\delta$ in \eqref{eq:42a}. For $f \in \GL_r (R [S^{-1}_p])$ choose a common denominator $s \in S_p$ for the entries of $f$ and set $g = sf \in M_r (R) \cap \GL_r (R [S^{-1}_p]) \subset M_r (R) \cap \GL_r (c_0 (\Gamma))$. Then $R^r / R^r g$ and $R^r / R^rs$ are both $S_p$-torsion. Namely choose $t \in S$ with $tf^{-1} \in M_r (R)$. Then $st \in S$ annihilates $R^r / R^r g$. We have
\[
\delta (f) = [R^r / R^r g] - [R^r / R^r s] \quad \text{in} \; K_0 (\Mh_{S_p} (R)) \; .
\]
Both $\log_p \det_{\Gamma}$ and $\tau^{\Gamma}_p$ are additive on $K_0 (\Mh_{S_p} (R))$ i.e. on short exact sequences of $p$-adically expansive $R$-modules. Hence it suffices to check equality in Proposition \ref{t5.4} for the modules $M_f = R^r / R^r f$ with $f \in M_r (R) \cap \GL_r (R [S^{-1}_p])$. For them we have $\cl_p [M_f] = [f]$ and hence
\[
\log_p \ddet_{\Gamma} [M_f] = \log_p \ddet_{\Gamma} f \; .
\]
On the other hand it is immediate from the definition of $\tau^{\Gamma}_p$ that we have
\[
\tau^{\Gamma}_p (M_f) = \log_p \ddet_{\Gamma} f \; .
\]
\end{proof}
\section{Euler characteristics and intersection theory} \label{sec:6}
In this section we set $\Gamma = \Z^N$ so that $\Z \Gamma = \Z [t^{\pm 1}_1 , \ldots , t^{\pm 1}_N]$ and we denote by $\Delta \subset \Gamma$ a subgroup of finite index with finite quotient group $G$. For suitable $\Z\Gamma$-modules $M$ we will express the multiplicative Euler-characteristic $\chi (\Delta , M)$ in terms of Serre's intersection numbers \cite{S}.

We call a $\Z\Gamma$-module $M$ expansive if it is finitely generated and if $L^1 \Gamma \otimes_{\Z\Gamma} M = 0$. This is equivalent to the algebraic dynamical system $X = \aM$ being expansive. Recall that the $\Z\Gamma$-module $M$ is $p$-adically expansive if and only if it is finitely generated and $c_0 (\Gamma) \otimes_{\Z\Gamma} M = 0$.

\begin{prop} \label{t6.1}
Let $M$ be an expansive or $p$-adically expansive $\Z\Gamma$-module. Then $\Z G \otimes_{\Z\Gamma} M$ is finite.
\end{prop}

\begin{proof}
Since $G$ is finite, the finitely generated $\Z G$-module $A = \Z G \otimes_{\Z\Gamma} M$ is finitely generated as an abelian group. The natural surjective ring homomorphisms
\[
L^1 \Gamma \twoheadrightarrow L^1 G = \C G \quad \text{and} \quad c_0 (\Gamma) \longrightarrow c_0 (G) = \Q_p G
\]
therefore induce surjections\[
L^1 \Gamma \otimes_{\Z\Gamma} M \twoheadrightarrow \C \otimes_{\Z} A \quad \text{and} \quad c_0 (\Gamma) \otimes_{\Z\Gamma} M \twoheadrightarrow \Q_p \otimes_{\Z} A \; .
\]
Thus ($p$-adic) expansiveness implies that $A$ is a torsion group, hence finite. 
\end{proof}

Set $\bbmu = \spec \Z G$ viewed as a closed subgroup scheme of $\G = \spec \Z\Gamma = \G^N_{m,\Z}$. The subgroup $\bbmu (\oQ) \subset \G (\oQ) = (\oQ^{\times})^N$ is finite and hence contained in $\mu (\oQ)^N$ where $\mu (\oQ)$ is the group of roots of unity in $\oQ^{\times}$. For $\Delta = n \Gamma = (n \Z)^N$ for example, we have $G = (\Z / n)^N$ and hence $\bbmu = \mu^N_{n, \Z}$ where $\mu_{n, \Z}$ is the group-scheme over $\Z$ of $n$-th roots of unity. Recall the affine scheme $\supp M = \spec (\Z \Gamma / \Ann (M))$.

\begin{prop}
\label{t6.2}
Let $M$ be a finitely generated $\Z\Gamma$-module. The group $ \Z G \otimes_{\Z\Gamma} M$ is finite if and only if $\bbmu (\oQ) \cap (\supp M) (\oQ) = \emptyset$. In this case the groups $\Tor^{\Z \Gamma}_i (\Z G, M)$ are finite as well and viewing them as $\Z\Gamma$--modules we have
\[
\Tor^{\Z\Gamma}_i (\Z G , M) = \bigoplus_{\emm} \Tor^{(\Z\Gamma)_{\emm}}_i ((\Z G)_{\emm} , M_{\emm}) \; .
\]
Here $\emm$ runs over the finitely many (closed) points of $\bbmu \cap \supp M$ in $\G$. 
\end{prop}

\begin{proof}
The group $\Z G \otimes_{\Z\Gamma} M$ is finitely generated as a $\Z G$-module and hence as a $\Z$-module. Hence it is finite if and only if $\Q G \otimes_{\Z\Gamma} M = 0$ or equivalently $S := \supp \Q G \otimes_{\Z\Gamma} M = \emptyset$. We have
\[
S = \bbmu_{\Q} \cap (\supp M)_{\Q} \quad \text{in} \; \G_{\Q} = \G \otimes_{\Z} \Q \; ,
\]
and this is empty if and only if
\[
S (\oQ) = \bbmu (\oQ) \cap (\supp M) (\oQ)
\]
is empty. Set $I = \Ker (\Z \Gamma \to \Z G)$. The ring $\Z\Gamma$ acts on $\Tor^{\Z \Gamma}_i (\Z G, M)$ via its quotient 
\[
\oR = \Z G \otimes_{\Z\Gamma} (\Z \Gamma / \Ann (M)) = \Z \Gamma / (I + \Ann (M)) \; .
\]
Since $\spec \oR = \bbmu \cap \supp M$ is finite over $\spec \Z$ and its generic fibre is empty by assumption, $\spec \oR$ is $0$-dimensional and consists of finitely many closed points. They correspond to the maximal ideals $\emm$ in $\Z\Gamma$ containing both $I$ and $\Ann (M)$. By the structure theorem for Artin algebras, we have
\[
\oR = \prod_{\emm} \oR_{\emm} \; .
\]
Since localization is exact, we get
\[
\Tor^{\Z\Gamma}_i (\Z G , M) = \Tor^{\Z \Gamma}_i (\Z G , M) \otimes_R \oR = \bigoplus_{\emm} \Tor^{(\Z\Gamma)_{\emm}}_i ((\Z G)_{\emm} , M_{\emm}) \; .
\]
The $\Tor$-groups are finite since they are finitely generated $\oR$-modules and the ring $\oR$ is finite. 
\end{proof}

\begin{rem}
Choosing embeddings $\oQ \subset \oQ_p$ and $\oQ \subset \C$ we have $\bbmu (\oQ) = \bbmu (\oQ_p) \subset T^N_p$ and $\mu (\oQ) = \mu (\C) \subset T^N$. Using the characterizations \eqref{eq:31} resp. \eqref{eq:32} of ($p$-adic) expansiveness, the first part of Proposition \ref{t6.2} gives a more geometric proof of Proposition \ref{t6.1}.
\end{rem}

In \cite{S} Serre developed a theory of local intersection numbers using higher $\Tor$'s. We review the main definitions and results.

Let $A$ be a regular local ring and $M_1 , M_2$ finite $A$-modules such that the length of $M_1 \otimes_A M_2$ is finite. Then the intersection multiplicity
\[
i (M_1 , M_2 ; A) = \sum^{\dim A}_{i=0} (-1)^i l_A (\Tor^A_i (M_1 , M_2))
\]
is well defined, and moreover $\dim M_1 + \dim M_2 \le \dim A$ where $\dim M = \dim\supp M$. Roberts \cite{R} and Gillet--Soul\'e \cite{GS} showed that we have
\begin{equation}
\label{eq:45}
i_A (M_1 , M_2 ; A) = 0 \quad \text{if} \; \dim M_1 + \dim M_2 < \dim A \; .
\end{equation}
Gabber proved non-negativity:
\begin{equation}
\label{eq:46}
i_A (M_1 , M_2 ; A) \ge 0 \; .
\end{equation}
Positivity of $i_A (M_1 , M_2 ; A)$ for $\dim M_1 + \dim M_2 = \dim A$ is not yet known in general. All these properties were conjectured by Serre and proven by him in several special cases, in particular for equicharacteristic or unramified local rings $A$.

\begin{cor}
\label{t6.3}
Let $M$ be a finitely generated $\Z\Gamma$-module, $\Gamma = \Z^N$ and let $\Delta \subset \Gamma$ be a subgroup of finite index with quotient group $G$. Assume that $\bbmu (\oQ) \cap (\supp M) (\oQ) = \emptyset$ where $\bbmu = \spec \Z G \subset \G$ or equivalently that $\Z G \otimes_{\Z\Gamma} M$ is finite. Then we have:
\begin{equation}
\label{eq:47}
\chi (\Delta , M) = \prod_{\emm} N \emm^{i ((\Z G)_{\emm} , M_{\emm} ; (\Z\Gamma)_{\emm})} \; .
\end{equation}
Here $\emm$ runs over those maximal ideals in $\bbmu \cap \supp M$ for which $\dim M_{\emm} = N$ and we write $N \emm = |\Z \Gamma / \emm|$. Moreover $\chi (\Delta , M) \ge 1$ is an integer.
\end{cor}

\begin{proof}
Let $F_{\hullet} \to M$ be a free resolution of the $\Z\Gamma$-module $M$. It is also a free resolution by $\Z\Delta$-modules. We have
\[
\Z G \otimes_{\Z\Gamma} F_{\hullet} = \Z \otimes_{\Z\Delta} F_{\hullet}
\]
canonically as complexes and therefore
\[
\Tor^{\Z\Gamma}_i (\Z G , M) = \Tor^{\Z\Delta}_i (\Z , M) = H_i (\Delta , M) \; .
\]
Formula \eqref{eq:47} now follows from Proposition \ref{t6.2} and assertions \eqref{eq:45}, \eqref{eq:46} noting that $\dim (\Z\Gamma)_{\emm} = N + 1$ for the maximal ideals $\emm$ of $\Z\Gamma$ and $\dim (\Z G)_{\emm} = 1$ for the closed points $\emm$ of $\bbmu$. By Hilbert's weak Nullstellensatz the field $\Z \Gamma / \emm$ is finite for every maximal ideal $\emm$ and hence $N \emm$ is defined. 
\end{proof}

\begin{rem}
Note that the conditions on $M$ in Corollary \ref{t6.3} are satisfied if $M$ is ($p$-adically) expansive. 
\end{rem}

For a finitely generated $\Z\Gamma$-module $M$ there is a filtration by submodules $0 = M_0 \subset M_1 \subset \ldots \subset M_n = M$ such that for $1 \le i \le n$ we have $M_i / M_{i-1} \cong \Z\Gamma / \ep_i$ with $\ep_i \in \supp M$. If $\bbmu (\oQ) \cap (\supp M) (\oQ) = \emptyset$ then $\bbmu (\oQ) \cap \supp (\Z \Gamma / \ep_i) (\oQ) = \emptyset$ as well, since
\[
\supp \Z \Gamma / \ep_i = V (\ep_i) \subset \supp M \; .
\]
Hence, if $\chi (\Delta , M)$ is defined, $\chi (\Delta , \Z\Gamma / \ep_i)$ is defined for all $i$ as well and we have
\begin{equation}
\label{eq:48}
\chi (\Delta , M) = \prod^n_{i=1} \chi (\Delta , \Z \Gamma / \ep_i)
\end{equation}
since $\chi (\Delta , \_) $ is multiplicative in short exact sequences.

\begin{prop}
\label{t6.4}
Let $\ep$ be a prime ideal in $\Z\Gamma$ with $\bbmu (\oQ) \cap V (\ep) (\oQ) = \emptyset$. If $\ep$ is not principal, then $\chi (\Delta , \Z \Gamma / \ep) = 1$. 
\end{prop}

\begin{proof}
Since $\Z\Gamma$ is a Noetherian unique factorization domain, the non-principal prime ideals $\ep$ have height $\tht (\ep) \ge 2$. The dimension inequality (in fact an equality)
\[
\dim \Z\Gamma / \ep + \tht (\ep) \le \dim \Z\Gamma = N + 1
\]
therefore gives
\[
\dim (\Z\Gamma / \ep)_{\emm} \le \dim \Z\Gamma / \ep \le N-1 \quad \text{for} \; \emm \in \supp \Z\Gamma / \ep \; .
\]
Now the assertion follows from formula \eqref{eq:47} applied to $M = \Z\Gamma / \ep$.
\end{proof}

\begin{cor}
\label{t6.5}
Let $M$ be as in Corollary \ref{t6.3} and let $\ep_i = (f_i)$ for $i \in I \subset \{ 1 , \ldots , n \}$ be the principal prime ideals in \eqref{eq:48}. Then we have
\[
\chi (\Delta , M) = \prod_{i\in I} \chi (\Delta , \Z \Gamma / f_i \Z\Gamma) \quad \text{where} \; \chi (\Delta , \Z\Gamma / f_i \Z\Gamma) = |\Z G/ f_i \Z G| \; .
\]
\end{cor}

\begin{proof}
The condition $\bbmu (\oQ) \cap \spec \Z\Gamma / \ep_i = \emptyset$ means that $\Z G \otimes_{\Z\Gamma} \Z \Gamma / \ep_i = \Z G / \ep_i \Z G$ is finite by Proposition \ref{t6.2}. For $\ep_i = (f_i)$ this shows that $f_i : \Q G \to \Q G$ is surjective and hence injective. Thus
\[
H_{\nu} (\Delta , \Z\Gamma/ f_i \Z \Gamma) = H_{\nu} (\Z G \xrightarrow{f_i} \Z G)
\]
is zero for $\nu \ge 1$ and equal to $\Z G / f_i \Z G$ for $\nu = 0$. Using Proposition \ref{t6.4} and formula \eqref{eq:48}, the assertion follows.
\end{proof}

We will now apply the preceding results to the calculation of $p$-adic entropy $h^{\per}_p$ as defined in formula \eqref{eq:40} or equivalently of $p$-adic $R$-torsion.

\begin{cor}
\label{t6.6}
Let $X$ be a $p$-adically expansive $\Gamma = \Z^N$-algebraic dynamical system and $M = \aX$ the corresponding $\Z\Gamma$-module. With notations as in Corollary \ref{t6.5} we have:
\[
h^{\per}_p = \tau^{\Gamma}_p (M) = \sum_{i \in I} \log_p \ddet_{\Gamma} f_i \; .
\]
In particular $h^{\per}_p = 0$ if all the prime ideals $\ep_i$ occuring in $M$ are non-principal.
\end{cor}

\begin{proof}
This follows from Theorem \ref{t5.3}, Propositions \ref{t6.1}, \ref{t6.2} and Corollary \ref{t6.5}
\end{proof}

\begin{rem}
For classical entropy we have the analogous formula
\[
h = \sum_{i \in I} \log \ddet_{\Nh\Gamma} f_i
\]
which holds even without expansiveness assumptions. See \cite[Corollary 18.5 and Proposition 18.6]{Sch} and use that the logarithm of the Fuglede--Kadison determinant $\log \det_{\Nh\Gamma} f_i$ equals the Mahler measure of $f_i$
\[
m (f_i) = \int_{T^N} \log |f_i| \, d\mu \; .
\]
Here $\mu$ is the Haar measure of $T^N$. This is proved using the Yuzvinskii addition formula for entropy and an argument based on the monotonicity of entropy. The latter is used to show that algebraic dynamical systems $X = (\Z \Gamma / \ep)^{\ast}$ for non-principal prime ideals have zero entropy. This argument does not transfer to $p$-adic entropy and we replaced it by the vanishing result \eqref{eq:45} for local intersection numbers.
\end{rem}

We now relate classical entropy with multiplicative Euler characteristics:

\begin{cor}
\label{t6.7}
Let $X$ be an expansive $\Gamma = \Z^N$-algebraic dynamical system. Then we have
\begin{equation}
\label{eq:49}
h = \lim_{n\to\infty} (\Gamma : \Gamma_n)^{-1} \log \chi (\Gamma_n , X) \quad \text{for} \; \Gamma_n \to 0 \; .
\end{equation}
\end{cor}

\begin{rem}
In \cite[VI 21]{Sch} the formula
\begin{equation}
\label{eq:50}
h = \lim_{n\to\infty} (\Gamma : \Gamma_n)^{-1} \log |X^{\Gamma_n}|
\end{equation}
is proved for expansive $X$. Thus $\chi (\Gamma_n , X)$ can be replaced by $|H^0 (\Gamma_n , X)|$ and the limit remains the same. Formula \eqref{eq:49} is more natural than formula \eqref{eq:50} because both the entropy and the logarithmic Euler characteristics are additive in short exact sequences of algebraic dynamical systems. On the other hand, formula \eqref{eq:50} is more explicit.
\end{rem}

\begin{proofof} {\it Corollary \ref{t6.7}} 
By additivity of entropy and its vanishing for non-principal actions, we have
\[
h = \sum_{i\in I} h_{f_i} \quad \text{where} \; h_f = \; \text{entropy of} \; X_f := (\Z \Gamma / f \Z \Gamma)^{\ast} \; .
\]
Using Proposition \ref{t5.2} and Corollary \ref{t6.5} we find
\[
\chi (\Gamma_n , X) = \prod_{i \in I} |X^{\Gamma_n}_{f_i}| \; .
\]
Combining this with formula \eqref{eq:50} applied to $X = X_{f_i}$:
\[
h_{f_i} = \lim_{n\to\infty} (\Gamma : \Gamma_n)^{-1} \log |X^{\Gamma_n}_{f_i}|
\]
wo obtain \eqref{eq:49}.
\end{proofof}

We end with a remark about a connection to Arakelov theory \cite{So}. Embedding $\G = \G^N_{m, \Z}$ into $\Pa = \Pa^N_{\Z}$, the closed subscheme $\bbmu$ of $\G$ is also closed in $\Pa$ since it is finite over $\spec \Z$. For $M = \Z\Gamma / \ea$ consider $V (\ea) = \spec \Z \Gamma / \ea$ and its closure $\overline{V (\ea)}$ in $\Pa$. We have $\bbmu \cap V (\ea) = \bbmu \cap \overline{V (\ea)}$ since $\bbmu$ is closed in $\Pa$ and therefore the right side of the formula for $\log \chi (\Delta , M)$ obtained from \eqref{eq:47} is the sum of all the non-archimedean local Arakelov intersection numbers of $\bbmu$ and $\overline{V (\ea)}$ in $\Pa^N$. 

For principal $\ea = (f)$ the global intersection pairing on $\Pa^N$ with the Arakelov cycle $(\ddiv f , - \log |f|^2)$ is zero. Hence $\log \chi (\Delta , M)$ then equals the negative of the sum over the local archimedian intersection numbers. We assume $M = \Z\Gamma / f \Z \Gamma$ to be expansive here, i.e. $f (z) = 0$ for all $z \in T^N$. \\
The Archimedian contribution to the global intersection number is $\sum_{\zeta \in \bbmu (\C)} \log |f (\zeta)|$. For $\Delta = \Gamma_n = (n \Z)^N$ we therefore find
\[
(\Gamma : \Gamma_n)^{-1} \log \chi (\Gamma_n , \Z \Gamma / f \Z \Gamma)) = n^{-N} \sum_{\zeta \in \mu_n (\C)^N} \log |f (\zeta)|
\]
and this converges for $n \to \infty$ to
\[
\int_{T^N} \log |f| \, d\mu = \log \ddet_{\Nh \Gamma} f = h \; ,
\]
as it has to.

\begin{thebibliography}{BLR08}

\bibitem[Bas68]{Ba}
Hyman Bass.
\newblock {\em Algebraic {$K$}-theory}.
\newblock W. A. Benjamin, Inc., New York-Amsterdam, 1968.

\bibitem[BLR08]{BLR}
Arthur Bartels, Wolfgang L\"uck, and Holger Reich.
\newblock On the {F}arrell-{J}ones conjecture and its applications.
\newblock {\em J. Topol.}, 1(1):57--86, 2008.

\bibitem[Bou70]{Bourbaki}
N.~Bourbaki.
\newblock {\em \'El\'ements de math\'ematique. {A}lg\`ebre. {C}hapitres 1 \`a
  3}.
\newblock Hermann, Paris, 1970.

\bibitem[Br{\"a}10]{B}
Jonas Br{\"a}uer.
\newblock Entropies of algebraic $\mathbb{Z}^d$-actions and {$K$}-theory.
\newblock Dissertation,
  https://miami.uni-muenster.de/Record/73542006-bab2-4d83-a4dc-d332a5a6c28e,
  2010.

\bibitem[CL15]{CL}
Nhan-Phu Chung and Hanfeng Li.
\newblock Homoclinic groups, {IE} groups, and expansive algebraic actions.
\newblock {\em Invent. Math.}, 199(3):805--858, 2015.

\bibitem[Coh73]{C}
Marshall~M. Cohen.
\newblock {\em A course in simple-homotopy theory}.
\newblock Springer-Verlag, New York-Berlin, 1973.
\newblock Graduate Texts in Mathematics, Vol. 10.

\bibitem[Den06]{D1}
Christopher Deninger.
\newblock Fuglede-{K}adison determinants and entropy for actions of discrete
  amenable groups.
\newblock {\em J. Amer. Math. Soc.}, 19(3):737--758, 2006.

\bibitem[Den09]{D}
Christopher Deninger.
\newblock {$p$}-adic entropy and a {$p$}-adic {F}uglede-{K}adison determinant.
\newblock In {\em Algebra, arithmetic, and geometry: in honor of {Y}u. {I}.
  {M}anin. {V}ol. {I}}, volume 269 of {\em Progr. Math.}, pages 423--442.
  Birkh\"auser Boston, Inc., Boston, MA, 2009.

\bibitem[DS07]{DS}
Christopher Deninger and Klaus Schmidt.
\newblock Expansive algebraic actions of discrete residually finite amenable
  groups and their entropy.
\newblock {\em Ergodic Theory Dynam. Systems}, 27(3):769--786, 2007.

\bibitem[FL03]{FL}
F.~Thomas Farrell and Peter~A. Linnell.
\newblock Whitehead groups and the {B}ass conjecture.
\newblock {\em Math. Ann.}, 326(4):723--757, 2003.

\bibitem[GS85]{GS}
Henri Gillet and Christophe Soul\'e.
\newblock {$K$}-th\'eorie et nullit\'e des multiplicit\'es d'intersection.
\newblock {\em C. R. Acad. Sci. Paris S\'er. I Math.}, 300(3):71--74, 1985.

\bibitem[Har77]{H}
Robin Hartshorne.
\newblock {\em Algebraic geometry}.
\newblock Springer-Verlag, New York-Heidelberg, 1977.
\newblock Graduate Texts in Mathematics, No. 52.

\bibitem[Kap69]{K}
Irving Kaplansky.
\newblock {\em Fields and rings}.
\newblock The University of Chicago Press, Chicago, Ill.-London, 1969.

\bibitem[LT14]{LT}
Hanfeng Li and Andreas Thom.
\newblock Entropy, determinants, and {$L^2$}-torsion.
\newblock {\em J. Amer. Math. Soc.}, 27(1):239--292, 2014.

\bibitem[Mat80]{M}
Hideyuki Matsumura.
\newblock {\em Commutative algebra}, volume~56 of {\em Mathematics Lecture Note
  Series}.
\newblock Benjamin/Cummings Publishing Co., Inc., Reading, Mass., second
  edition, 1980.

\bibitem[Mil66]{Mi}
J.~Milnor.
\newblock Whitehead torsion.
\newblock {\em Bull. Amer. Math. Soc.}, 72:358--426, 1966.

\bibitem[Rob85]{R}
Paul Roberts.
\newblock The vanishing of intersection multiplicities of perfect complexes.
\newblock {\em Bull. Amer. Math. Soc. (N.S.)}, 13(2):127--130, 1985.

\bibitem[Sch95]{Sch}
Klaus Schmidt.
\newblock {\em Dynamical systems of algebraic origin}, volume 128 of {\em
  Progress in Mathematics}.
\newblock Birkh\"auser Verlag, Basel, 1995.

\bibitem[Ser00]{S}
Jean-Pierre Serre.
\newblock {\em Local algebra}.
\newblock Springer Monographs in Mathematics. Springer-Verlag, Berlin, 2000.
\newblock Translated from the French by CheeWhye Chin and revised by the
  author.

\bibitem[Sou92]{So}
C.~Soul\'e.
\newblock {\em Lectures on {A}rakelov geometry}, volume~33 of {\em Cambridge
  Studies in Advanced Mathematics}.
\newblock Cambridge University Press, Cambridge, 1992.
\newblock With the collaboration of D. Abramovich, J.-F. Burnol and J. Kramer.

\bibitem[Swa78]{Sw}
Richard~G. Swan.
\newblock Projective modules over {L}aurent polynomial rings.
\newblock {\em Trans. Amer. Math. Soc.}, 237:111--120, 1978.

\bibitem[Wei94]{W}
Charles~A. Weibel.
\newblock {\em An introduction to homological algebra}, volume~38 of {\em
  Cambridge Studies in Advanced Mathematics}.
\newblock Cambridge University Press, Cambridge, 1994.

\end{thebibliography}

\end{document}